\documentclass[12pt,twoside]{amsart}
\usepackage[utf8]{inputenc} 
\usepackage[T1]{fontenc}

\usepackage{amsmath,amssymb,amsfonts,amsthm,mathtools}

\usepackage{graphicx} 
\usepackage{tikz} \usetikzlibrary{arrows.meta,calc} 
\usepackage{tikz-cd} 
\usepackage[all]{xy} 
\usepackage{subcaption} 

\usepackage{array} 
\usepackage{booktabs} 
\usepackage{tabularx} 
\usepackage{colortbl} 

\usepackage{xcolor} 

\usepackage{thmtools} 

\usepackage{leftidx} 
\usepackage{enumitem} 
\usepackage{verbatim} 
\usepackage[normalem]{ulem} 
\usepackage{geometry} 

\usepackage[colorlinks=true]{hyperref}

\usetikzlibrary{patterns, decorations.pathreplacing}

\numberwithin{equation}{section}

\DeclareMathOperator{\supp}{supp}
\DeclareMathOperator{\FB}{FB}
\DeclareMathOperator{\ev}{ev}
\DeclareMathOperator{\Span}{Span}
\DeclareMathOperator{\MW}{MW}

\newcommand{\F}{\mathbb{F}}

\newcommand{\Z}{\mathbb{Z}}

\theoremstyle{plain} 
\newtheorem{theorem}{Theorem}[section]
\newtheorem{lemma}[theorem]{Lemma}
\newtheorem{proposition}[theorem]{Proposition}
\newtheorem{corollary}[theorem]{Corollary}

\theoremstyle{definition}
\newtheorem{definition}[theorem]{Definition} 
 
\newtheorem{remark}[theorem]{Remark}
\newtheorem{example}[theorem]{Example}

\title[Polynomials with restricted support and maximal zeros]{Polynomials with restricted support and maximal zeros on a finite Cartesian set}

\author[Bhunia]{Dipak K. Bhunia}
\address{Dipak K. Bhunia\\
Universitat Polit\`ecnica de Catalunya,
Barcelona, Spain}
\email{dipak.kumar.bhunia@upc.edu}

\author[Camps-Moreno]{Eduardo Camps-Moreno}
\address{Eduardo Camps-Moreno\\
Universit\'e de Bordeaux, Talence, France}
\email{eduardo.camps-moreno@math.u-bordeaux.fr}

\author[García-Marco]{Ignacio García-Marco}
\address{Ignacio García-Marco\\
Instituto de Matem\'aticas y Aplicaciones (IMAULL), Secci\'on de Matem\'aticas, Facultad de
Ciencias, Universidad de La Laguna, Spain}
\email{iggarcia@ull.edu.es}

\author[L\'opez]{Hiram H. L\'opez}
\address{Hiram H. L\'opez\\
Virginia Tech,
Blacksburg, VA, USA}
\email{hhlopez@vt.edu}

\author[Márquez-Corbella]{Irene Márquez-Corbella}
\address{Irene Márquez-Corbella\\
Instituto de Matem\'aticas y Aplicaciones (IMAULL), Secci\'on de Matem\'aticas, Facultad de
Ciencias, Universidad de La Laguna, Spain}
\email{imarquec@ull.edu.es}

\thanks{D. K. Bhunia is supported by the Juan de la Cierva 2024 grant (JDC2024-053082-I), funded by Ministerio de Ciencia, Innovación y Universidades (MICIU/AEI/10.13039/501100011033) and co-funded by the European Social Fund Plus (FSE+).  \\  I. Garc\'ia-Marco  and I. M\'arquez-Corbella are supported by the project MICINN PID2023-149508NB-I00, funded by MICIU/AEI/10.13039/501100011033 and by FEDER, EU. \\ D. K. Bhunia, I. Garc\'ia-Marco and I. M\'arquez-Corbella were supported through the Talentum program by the Cabildo Insular de Tenerife in collaboration with the Universidad de La Laguna. \\ H. H. L\'opez was partially supported by the Commonwealth Cyber Initiative and by the NSF grants DMS-2401558 and DMS-2502705.}
\begin{document}

\begin{abstract}
Given a finite Cartesian set $S=X \times Y$ and a decreasing set of monomials $\mathcal M$, we call extremal polynomials for $\mathcal M$ over $S$ those that have the maximum number of zeros in $S$ and whose support belongs to $\mathcal M$. Coordinate factorizations give a family of extremal polynomials; we call them canonical. If $\max(\mathcal M)$, taken with respect to divisibility, is a single monomial, all extremal polynomials are canonical. If $|\max(\mathcal M)|=2$, either all extremal polynomials are canonical, or the problem reduces to the case where $\max(\mathcal M)=\{x^{d_x},y^{d_y}\}$. In the latter case, we prove that the existence of noncanonical extremal polynomials depends on finding families of subsets whose elementary symmetric functions agree. This condition is more restrictive than the classical Prouhet--Tarry--Escott problem, which asks for two sets whose elementary symmetric functions agree. We determine the number of triples $(X, Y, h)$, where $h$ is a quadratic noncanonical extremal polynomial. We apply extremal polynomials to coding theory via minimum-weight codewords.
\end{abstract}

\maketitle

\section{Introduction}
Let $\F$ be a field and let $S=X\times Y\subseteq \F^2$ be a finite Cartesian product. Given a polynomial $f\in \F[x,y]$, a classical problem is to determine the maximum number of zeros that $f$ can have on $S$. While classical bounds are expressed in terms of the total degree of $f$
\cite{AlonFuredi93,DeMilloLipton78,Schwartz80,Zippel79}, the {\it footprint bound}~\cite{geilfp} replaces the degree by the leading monomial. Geil~\cite{Geil2026} has recently argued that this is the natural
multivariate analog of the factor theorem. Footprint techniques have been used to determine the
maximum number of zeros of polynomials in other settings, such as weighted
projective spaces over finite fields~\cite{NardiSanJose2026}. 

In this paper we work in the affine setting $S=X\times Y$ and study the problem
under support restrictions: given a decreasing set of monomials $\mathcal M$ (i.e., closed
under divisibility), we consider the polynomials whose support is contained in $\mathcal M$.
The footprint bound yields an upper bound $z_S(\mathcal M)$ for their number of zeros on $S$,
which is attained (Proposition~\ref{prop.notempty}). We call extremal polynomials for $\mathcal M$ over $S$
those attaining it, and we denote by $Z_S(\mathcal M)$ the set of them.  Thus, our goal is to understand the structure of the elements of $Z_S(\mathcal M)$.

A family of extremal polynomials is immediately obtained from coordinate factorizations of the form
\[
c\prod_{\alpha\in X'}(x-\alpha)\prod_{\beta\in Y'}(y-\beta),
\]
with $c \in \F \setminus \{0\}$, $X' \subseteq X$ and $Y' \subseteq Y$; whose zero sets are unions of rows and columns of the grid. We denote by $\Pi_S(\mathcal M)$ the set of extremal polynomials of this type supported on $\mathcal M$. A natural question is whether every polynomial in $Z_S(\mathcal M)$ belongs to $\Pi_S(\mathcal M)$, and if not, how to describe the additional extremal polynomials that arise. We call the elements of $\Pi_S(\mathcal M)$ canonical, and those of $Z_S(\mathcal M) \setminus \Pi_S(\mathcal M)$ noncanonical.

Our work builds on structural results of Geil~\cite{geil} concerning polynomials attaining the footprint bound (see also~\cite{Geil2026} for a unified treatment of common roots of sets of polynomials). Denoting by $\max(\mathcal M)$ the set of maximal elements of $\mathcal M$ with respect to divisibility, we use Geil's results to prove that if $|\max(\mathcal M)|=1$, then all extremal polynomials are canonical (Proposition~\ref{cor.onemon}). The same holds if $|\max(\mathcal M)|=2$ and the two maximal monomials have different footprint values (Theorem~\ref{thm.diffoot}). If they have the same footprint value, after removing common coordinate factors the problem reduces to $\max(\mathcal M)=\{x^{d_x},y^{d_y}\}$ with $|X|d_y=|Y|d_x$ (Proposition~\ref{prop.factorh}), which is the first genuinely nontrivial situation in the study of extremal polynomials on Cartesian grids. In this case, noncanonical extremal polynomials can exist only if $t=|X|/d_x=|Y|/d_y$ is an integer, a condition that is necessary but not sufficient (see Example~\ref{ex.11}). Moreover, every such polynomial has the form $h(x,y)=f(x)+g(y)$, the sets $X$ and $Y$ decompose into blocks $X=X_1\sqcup\dots\sqcup X_t$, $Y=Y_1\sqcup\dots\sqcup Y_t$ with $|X_i|=d_x$, $|Y_i|=d_y$, and $f$ (resp.\ $g$) is constant equal to $\lambda_i$ (resp.\ $-\lambda_i$) on $X_i$ (resp.\ $Y_i$), with the $\lambda_i$ pairwise distinct.

A second contribution of the paper is the connection between this extremal-zero problem and the Prouhet--Tarry--Escott  (PTE) problem, a classical problem in Diophantine analysis that asks for disjoint multisets of integers whose power sums agree up to a prescribed degree~\cite{BorweinIngalls94, Dickson52}.  Indeed, the block decomposition obtained above implies that the associated subsets satisfy equalities between elementary symmetric functions, and consequently equalities between power sums when the characteristic is zero or sufficiently large. However, our problem is more restrictive: the classical PTE problem allows multisets (with repeated elements), whereas each subset $X_i$ in our setting must consist of pairwise distinct elements. The existence of solutions to the PTE equations therefore does not by itself guarantee the existence of noncanonical extremal polynomials (see Example~\ref{ex.11}). Despite its long history, no systematic methods exist to construct ideal solutions to the PTE problem~\cite{BorweinLisonekPercival03, PTE-survey}. We show that classifying noncanonical extremal polynomials can be reformulated in terms of the existence of finite-field analogs of Prouhet--Tarry--Escott configurations.

After establishing the general structure theory, we focus on the first nontrivial degrees, namely $d_x,d_y\leq 2$. In this setting, the problem becomes sufficiently rigid to allow explicit enumeration. We determine the admissible triples $(X, Y,h)$ for which $h(x,y)=f(x)+g(y)$ attains the maximum possible number of zeros, and we obtain exact counting formulas. The analysis naturally splits into characteristic $2$ and odd characteristic. In characteristic $2$, the linear structure of quadratic maps governs the description. In contrast, in odd characteristic, the computation reduces to intersections of images of quadratic polynomials and can be expressed in terms of quadratic residues and character sums.

Extremal polynomials play an important role in algebraic coding theory, where the number of zeros of a polynomial on a finite set is directly related to the Hamming weight of the corresponding evaluation codeword. Minimum-weight codewords are of particular interest: they realize the minimum distance of the code and their number governs the leading term of the standard estimates of the decoding error probability~\cite{LinCostello04, SasonShamai06}, so knowing their structure and multiplicity is essential to assess the performance of a code. Moreover, their explicit description often reveals a rich geometric structure, such as hyperplane arrangements~\cite {BalletRolland14, RM}. Delsarte, Goethals, and MacWilliams characterized the minimum-weight codewords of Reed--Muller codes~\cite {RM}; see also~\cite{Leducq12} for an alternative proof. Carvalho and Neumann characterized affine Cartesian codes when the sets are nested subfields of $\mathbb{F}_q$~\cite{CarvalhoNeumann20}, and Dang and Ghorpade gave the corresponding enumeration~\cite {enumNAC}. In this work, we study the minimum-weight codewords of monomial decreasing evaluation codes
\[
\mathcal C(\F_q^2,\mathcal M).
\]
For such codes, minimum-weight codewords correspond exactly to polynomials attaining the maximum number of zeros on the underlying grid. Consequently, the structural and enumerative results developed in the first part of the paper translate into descriptions and counting formulas for minimum-weight codewords. In particular, we characterize the minimum-weight codewords of decreasing codes with at most two maximal monomials and obtain explicit formulas in the quadratic cases considered above.

The paper is organized as follows. Section~\ref{sec:preliminaries} introduces the basic setting and reviews the footprint bound and the structural results needed throughout. Section~\ref{sec:reduction} develops the reduction theory and reduces the problem to decreasing sets with two maximal monomials. Section~\ref{sec:structure} studies this reduced situation and proves the existence and structure theorems for extremal polynomials. Section~\ref{sec:PTE} connects them to Prouhet--Tarry--Escott configurations. Section~\ref{sec:quadratic} focuses on the quadratic case and enumerates extremal polynomials over finite fields. Finally, Section~\ref{sec:codes} applies these results to monomial decreasing codes and determines the corresponding minimum-weight codewords.

\section{Preliminaries}\label{sec:preliminaries}

Throughout this paper, $\F$ denotes a field and
$S=X\times Y\subseteq\F^2$
is a finite Cartesian product with $n_x=|X|$ and $n_y=|Y|$. For a polynomial $f\in\F[x,y]$, we denote by
$Z_S(f) = \{P\in S : f(P)=0\}$ its zero set on $S$.
The vanishing ideal of $S$, defined as the set of polynomials that vanish on $S$, is (see, e.g., \cite[Lemma 2.3]{LRV14})
\[
I_S=
\left(
\prod_{\alpha\in X}(x-\alpha),
\;
\prod_{\beta\in Y}(y-\beta)
\right)
\subseteq\F[x,y].
\]
We also consider the set
\[
R_S
=
\left\{
f\in\F[x,y]:
\deg_x(f)<n_x,\;
\deg_y(f)<n_y
\right\},
\]
which is a complete system of representatives for
$\F[x,y]/I_S$. Consequently, every polynomial $f\in\F[x,y]$
has a unique representative $\overline f\in R_S$
satisfying
\[
f\equiv\overline f\pmod{I_S}.
\]
Since every element of $I_S$ vanishes on $S$, the polynomials
$h$ and $\overline h$ have the same zero set on $S$.
Therefore, when studying zeros on $S$, it is enough
to work with polynomials in $R_S$.
 By the footprint bound \cite{geilfp,yuriko},
if $f\in R_S$ and $<$ is a monomial order, then
\[
|Z_S(f)|\le \dim \F[x,y]/\bigl(in_<(I_S),in_<(f)\bigr).
\]
As $in_<(I_S)=(x^{n_x},y^{n_y})$ for every monomial order, we get
\[
\dim\F[x,y]/(x^{n_x},y^{n_y},x^ay^b)=n_xb+n_ya-ab
\]
for every $0 \leq a < n_x,\, 0 \leq b < n_y$. The lattice points in Figure~\ref{fig:footprint} represent the standard monomials of $\F[x,y]/(x^{n_x},y^{n_y},x^a y^b)$.

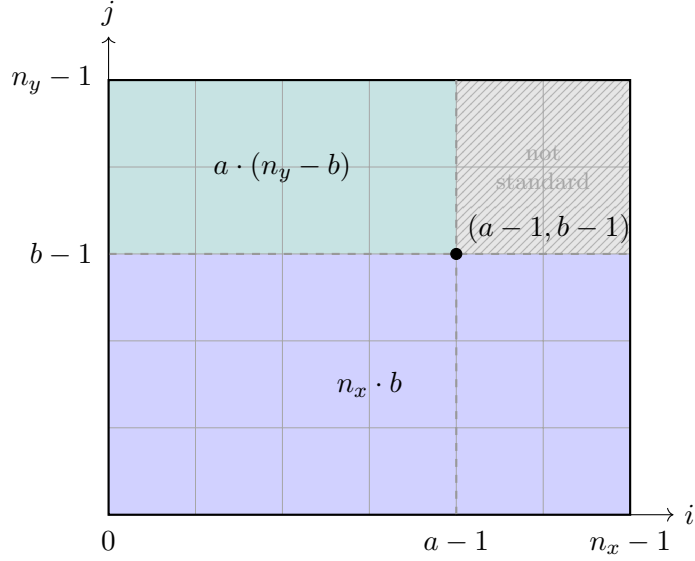
\begin{figure}[ht]
\centering
\begin{tikzpicture}[scale=1.15]
  \fill[blue!18] (0,0) rectangle (6,3);
  \fill[teal!22] (0,3) rectangle (4,5);
  \fill[gray!20] (4,3) rectangle (6,5);
  \fill[pattern=north east lines, pattern color=gray!60] (4,3) rectangle (6,5);

  \draw[very thin, gray!70] (0,0) grid (6,5);
  \draw[thick] (0,0) rectangle (6,5);
  \draw[dashed, thick, gray!80] (4,0) -- (4,5);
  \draw[dashed, thick, gray!80] (0,3) -- (6,3);
 
  \node[font=\small] at (3.0, 1.5) {$n_x \cdot b$};
  \node[font=\small] at (2.0, 4.0) {$a \cdot (n_y - b)$};
  \node[gray!70, font=\scriptsize, align=center] at (5.0, 4.0)
        {not\\standard};
  \node[below=2pt, font=\small] at (0,0) {$0$};
  \node[below=2pt, font=\small] at (4,0) {$a-1$};
  \node[below=2pt, font=\small] at (6,0) {$n_x-1$};
  \node[left=2pt, font=\small] at (0,3) {$b-1$};
  \node[left=2pt, font=\small] at (0,5) {$n_y-1$};
  
  \foreach \x in {0,4,6}
    \draw (\x, 0) -- ++(0,0);
  \foreach \y in {0,3,5}
    \draw (0,\y) -- ++(0,0);

  \filldraw (4,3) circle (1.8pt)
        node[above right, font=\small] {$(a-1,b-1)$};

  \draw[->] (0, 0) -- (6.5, 0)
        node[right, font=\small] {$i$};
  \draw[->] (0, 0) -- (0, 5.5)
        node[above, font=\small] {$j$};

\end{tikzpicture}
\caption{Standard monomials of $\F[x,y]/(x^{n_x},y^{n_y},x^a y^b)$ correspond to the lattice points in the L-shaped region. The bottom rectangle contributes $n_x b$ monomials and the upper-left rectangle contributes $a(n_y-b)$, giving
$\dim \F[x,y]/(x^{n_x},y^{n_y},x^a y^b) = n_x b + n_y a - ab$.}
\label{fig:footprint}
\end{figure}

The previous equality involving the dimension of the quotient ring motivates the following definition.
\[
z_S(a,b)=n_xb+n_ya-ab.
\]
Hence, for every monomial order $<$, $|Z_S(f)| \le z_S(a,b)$ whenever $in_{<}(f)=x^ay^b$. In other words,  if for a polynomial $f\in\F[x,y]$, 
\[in(f)=\{in_{<}(f): < \text{ a monomial order}\}
\]
denotes the set of all possible leading monomials of $f$, then 
\[ |Z_S(f)| \le \min\{ z_S(a,b): x^ay^b\in in(f) \}.\]
\begin{definition}
A nonzero polynomial $f\in R_S$ has \emph{maximal zeros on $S$} if
\[ |Z_S(f)| = \min\{ z_S(a,b): x^ay^b\in in(f) \}. \]
\end{definition}
The simplest examples of polynomials that have maximal zeros are products of coordinate factors
\[
\prod_{\alpha\in X'}(x-\alpha)
\prod_{\beta\in Y'}(y-\beta),
\]
where $X'\subseteq X$ and $Y'\subseteq Y$.
A key tool in our work is the following factor-removal result.

\begin{lemma}{\cite[Lemma 3]{geil}}\label{lemma.factors} Let $f\in R_S$ be such that
$$f(x,y)=h(x,y)\prod_{\alpha\in X'}(x-\alpha)\prod_{\beta\in Y'}(y-\beta)$$
where $X'\subseteq X$ and $Y'\subseteq Y$. Then, $f$ has maximal zeros on $S$ if and only if $h$ has maximal zeros on $(X\setminus X')\times (Y\setminus Y')$. 
\end{lemma}

The next theorem, due to Geil, gives a strong restriction
on the possible leading monomials of a polynomial with maximal
zeros.

\begin{theorem}{\cite[Theorem 5]{geil}}\label{th.maingeil}
Let $f\in R_S$ be with maximal zeros on $S$. Define
\[
M=\{x^ay^b\in\mathrm{supp}(f)\ :\ z_S(a,b)=|Z_S(f)|\}.
\]
If $(x-\alpha)\not{|}f$ and $(y-\beta)\not{|}f$ for any $\alpha\in X$, $\beta\in Y$, then $in(f)\subseteq M$ and $\gcd(M)=1$.
\end{theorem}

Motivated by these results, we study polynomials whose support
is contained in a decreasing set of monomials.

\begin{definition}
Let $\mathcal M\subseteq \Delta(I_S)
=\{x^ay^b:0\le a<n_x,\;0\le b<n_y\}$ be a decreasing monomial set, meaning that if $u\in\mathcal M$ and $v$ divides $u$, then $v\in\mathcal M$. We define
\begin{itemize}
\item $z_S(\mathcal M)=
\max\{z_S(a,b):x^ay^b\in\mathcal M\}.$
\item $Z_S(\mathcal M)= \left\{ f\in R_S: \supp(f)\subseteq\mathcal M,\; |Z_S(f)|=z_S(\mathcal M) \right\}.$
\end{itemize}
The elements of $Z_S(\mathcal M)$ are called extremal polynomials for $\mathcal M$ on $S$.
\end{definition}
Observe that $Z_S(f)$ denotes a subset of the grid, whereas $Z_S(\mathcal M)$ denotes the family of extremal polynomials.

The structure of extremal polynomials $Z_S(\mathcal M)$ has been studied for some particular choices of $\mathcal M$ due to their applications. For example, in coding theory, extremal polynomials have been characterized when $\mathcal M$ consists of all monomials of degree at most $d$; see the following result.

\begin{theorem}{\cite[Theorem 2.6.3]{RM}}\label{thm:minRM}
Take $S=\F_q^2$ and the decreasing set
\[ \mathcal M = \left\{ x^ay^b: 0\leq a,b\leq q-1,\; a+b\leq d \right\}, \]
where $1\leq d\leq 2(q-1)-1$. Write $d=k(q-1)+\ell$, $1\leq\ell\leq q-1$.
Then,
\[ Z_S(\mathcal M) =
\bigcup_{\substack{X'\subseteq\F_q\\ |X'|=k(q-1)}} \; \bigcup_{\substack{Y'\subseteq\F_q\\ |Y'|=\ell}} \mathcal P(X',Y'),
\]
where
$$\mathcal P(X',Y') = 
\left\{
\prod_{\alpha\in X'} (\lambda_1x+\lambda_2y-\alpha) \prod_{\beta\in Y'} (\lambda_3x+\lambda_4y-\beta): \lambda_i\in\F_q, \lambda_1\lambda_4\neq\lambda_2\lambda_3
\right\}.$$
\end{theorem}

\begin{remark}
Note that if $f \in Z_S(\mathcal M)$, then $f$ has maximal zeros on $S$. 
Indeed, by the footprint bound, and as $\supp(f)\subseteq\mathcal M$, we get
$$|Z_S(f)| \leq \min\{ z_S(a,b): x^ay^b\in in(f) \}
\leq \max\{z_S(a,b):x^ay^b\in\mathcal M\}=z_S(\mathcal M).$$
Thus, $|Z_S(f)| = \min\{ z_S(a,b): x^ay^b\in in(f) \}$, which means that $Z_S(\mathcal M)$ is the set of polynomials attaining the largest possible number of zeros among all polynomials supported on $\mathcal M$.
\end{remark}
\begin{definition}
For a set of monomials $\mathcal M$, we denote by
$$ \max(\mathcal M) $$
the set of monomials in $\mathcal M$ that are maximal with respect to divisibility.
\end{definition}
We show below that a distinguished family of elements of $Z_S(\mathcal M)$ arises from coordinate factorizations. The main objective of this paper is to describe the structure of the sets $Z_S(\mathcal M)$ and, in particular, to determine when they contain elements outside this canonical family. The set $\max(\mathcal M)$ plays an important role.

\section{Reduction to the Two-Monomial Case}
\label{sec:reduction}

We continue with the notation from Section~\ref{sec:preliminaries}. In particular, $S=X\times Y\subseteq \F^2$ is a finite Cartesian product with $|X|=n_x$ and $|Y|=n_y$, and $\mathcal M\subseteq \Delta(I_S)$ is a decreasing set. In this section, we begin studying the sets $Z_S(\mathcal M)$. Our goal is to identify the first genuinely nontrivial support configurations. We show that several classes of decreasing sets lead only to the expected family of polynomials obtained from products of coordinate factors. Combining these reductions with Geil's structural theorem, we eventually reduce the problem to decreasing sets whose maximal monomials are two relatively prime monomials of equal footprint value.

\subsection{Basic properties}
First, $Z_S(\mathcal{M})$ is nonempty when $\mathcal{M}$ is decreasing.

\begin{proposition}\label{prop.notempty}
Let $\mathcal{M}\subseteq\Delta(I_S)$ be decreasing. Then $Z_S(\mathcal{M})\neq\emptyset$. Even more, if $z_S(\mathcal{M}) = z_S(a,b)$ with $x^a y^b \in \mathcal{M}$, then
$$f(x,y) = c\prod_{\alpha\in X'}(x-\alpha)\prod_{\beta\in Y'} (y-\beta) \in Z_S(\mathcal{M})$$
for $c\in\F^*$ and any $X'\subseteq X$ and $Y'\subseteq Y$ such that $|X'|=a$ and $|Y'|=b$.
\end{proposition}
\begin{proof} Since $z_S(\mathcal M)=z_S(a,b)$, we have $x^ay^b\in\mathcal M$. Every monomial in the support of $f$ divides $x^ay^b$, so the decreasing property of $\mathcal M$ implies that
$\supp(f)\subseteq\mathcal M$.
Moreover, the zero set of $f$ is the union $(X'\times Y)\cup(X\times Y')$. Therefore,
\[ |Z_S(f)| = n_y|X'|+n_x|Y'|-|X'||Y'|
= n_ya+n_xb-ab = z_S(a,b) = z_S(\mathcal M).
\]
Hence, $f\in Z_S(\mathcal M)$.
\end{proof}

Motivated by Proposition~\ref{prop.notempty}, we introduce the following family.
\begin{definition}
Let $\mathcal M\subseteq\Delta(I_S)$ be a decreasing set. We define
\[
\Pi_S(\mathcal M)=
\left\{
c\prod_{\alpha\in X'}(x-\alpha)\prod_{\beta\in Y'}(y-\beta)
\in Z_S(\mathcal M):
c\in\F^*,\;X'\subseteq X,\;Y'\subseteq Y
\right\}.
\]
\end{definition}
The set $\Pi_S(\mathcal M)$ provides a distinguished family of elements in $Z_S(\mathcal M)$. The elements of 
$$ Z_S(\mathcal M)\setminus\Pi_S(\mathcal M) $$
are called noncanonical extremal polynomials for $\mathcal M$ on $S$. Our objective is to determine when noncanonical extremal polynomials exist.

\begin{theorem}{\cite[Theorem 7]{geil}}\label{th.onemonomial}
Let $\mathcal M\subseteq\Delta(I_S)$ be decreasing. Take $f\in R_S$ and let $x^ay^b\in in(f)$ be such that $z_S(a,b)>z_S(c,d)$ for any $x^cy^d\in \mathrm{supp}(f)$ with $(a,b)\neq (c,d)$. Then,
\[ f\in\Pi_S(\mathcal M) \qquad\text{or}\qquad |Z_S(f)|<z_S(\mathcal M). \]
\end{theorem}

\subsection{The case of a unique maximal monomial}

We begin with the simplest case: $\max(\mathcal{M})=\{x^ay^b\}$.
\begin{lemma}\label{prop.onegen}
Let $\mathcal M\subseteq\Delta(I_S)$ be decreasing.
Let $f\in R_S$ satisfy $\supp(f)\subseteq\mathcal M$. Assume that
\[
\bigl|\{x^ay^b\in\supp(f):z_S(a,b)=z_S(\mathcal M)\}\bigr|=1.
\]
Then, $f\in Z_S(\mathcal M) \Longleftrightarrow f\in\Pi_S(\mathcal M)$.
\end{lemma}
\begin{proof}
This follows from Theorem~\ref{th.onemonomial}.
\end{proof}

As an immediate consequence, when $\mathcal M$ has a unique maximal monomial, every extremal polynomial belongs to $\Pi_S(\mathcal M)$.

\begin{proposition}\label{cor.onemon}
     Let $\mathcal{M}\subseteq\Delta(I_S)$ be decreasing with $\max(\mathcal{M})=\{x^ay^b\}$. Then, 
    \[Z_S(\mathcal{M})=\Pi_S(\mathcal{M}).\]
\end{proposition}
\begin{proof}
Let $f\in R_S$ with $\supp(f)\subseteq\mathcal{M}$. For every $u=x^cy^d\in in(f)$, we have $u\mid x^ay^b$. If $u\neq x^ay^b$, then $c\leq a$, $d\leq b$, with at least one inequality strict, and hence
\[
z_S(c,d)<z_S(a,b)=z_S(\mathcal M).
\]
Therefore, we would have that $|Z_S(f)|<z_S(\mathcal M)$. Thus, if $f\in Z_S(\mathcal M)$, we must have $in(f)=\{x^ay^b\}.$ The conclusion now follows from Lemma~\ref{prop.onegen}.
\end{proof}

Therefore, when $|\max(\mathcal M)|=1$, the extremal polynomials are characterized by products of coordinate factors.

\subsection{Two maximal monomials with different footprint values}
We next consider the case 
\[
\max(\mathcal M)=
\{x^{a_1}y^{b_1},x^{a_2}y^{b_2}\}
\qquad \text{ and } \qquad
z_S(a_1,b_1) \neq z_S(a_2,b_2).
\]
When the two maximal monomials have different footprint values, the previous argument still determines all extremal polynomials, as the following result shows.

\begin{theorem}\label{thm.diffoot}
Let $\mathcal{M}\subseteq\Delta(I_S)$ be decreasing with $\max(\mathcal M)=\{x^{a_1}y^{b_1},x^{a_2}y^{b_2}\}$ and $z_S(a_1,b_1) \neq z_S(a_2,b_2)$. Then,
    $$Z_S(\mathcal{M})=\Pi_S(\mathcal{M}).$$
\end{theorem}

\begin{proof}
    Let $f\in R_S$ with $\supp(f)\subseteq\mathcal{M}$. Let $u\in in(f)$. We have that $u|x^{a_1}y^{b_1}$ or $u|x^{a_2}y^{b_2}$. If $u\neq x^{a_2}y^{b_2}$, then 
    $$|Z_S(f)|<z_S(a_2,b_2),$$

    \noindent and thus $f\notin Z_S(\mathcal M)$. If $u=x^{a_2}y^{b_2}$, then by Theorem \ref{th.onemonomial}, $|Z_S(f)|=z_S(a_2,b_2)$ if and only if $f$ has the desired form.
\end{proof}

Hence, whenever there are only two maximal monomials, and they have different footprint values, the set $Z_S(\mathcal M)$ is completely described by $\Pi_S(\mathcal M)$.

\subsection{Two maximal monomials with equal footprint values}
Assume now that
\[
\max(\mathcal M)=
\{x^{a_1}y^{b_1},x^{a_2}y^{b_2}\}
\qquad \text{ and } \qquad
z_S(a_1,b_1) = z_S(a_2,b_2).
\]
In the current situation, any polynomial $f\in Z_S(\mathcal M)$ must satisfy
$in(f)\cap \max(\mathcal M) \neq \emptyset$. Moreover, if only one maximal monomial appears as a possible leading term, then $Z_S(\mathcal M)$ is again completely described by $\Pi_S(\mathcal M)$, as the following result shows.
\begin{proposition}\label{prop.1init}
Let $\mathcal M\subseteq\Delta(I_S)$ be decreasing with
$\max(\mathcal M) = \{x^{a_1}y^{b_1},x^{a_2}y^{b_2}\}$ and $z_S(a_1,b_1)=z_S(a_2,b_2)$. If $f\in Z_S(\mathcal M)$, then
\[ f\in\Pi_S(\mathcal M)
\quad\Longleftrightarrow\quad
|in(f)\cap\max(\mathcal M)|=1.\]
\end{proposition}
\begin{proof}
This follows from Lemma~\ref{prop.onegen}.
\end{proof}

The only remaining case is when $in(f)=\max(\mathcal M)$. The next result shows that a common divisor among the maximal monomials forces the existence of coordinate factors.

\begin{proposition}\label{prop.linearfactors}
Assume that $\gcd(x^{a_1}y^{b_1},x^{a_2}y^{b_2})\neq 1$ and take $f\in R_S$ with $\supp(f)\subseteq\mathcal{M}$. If $f$ has no factor of the form $x-\alpha$ or $y-\beta$ for any $\alpha\in X$, $\beta\in Y$, then $f\notin Z_S(\mathcal{M})$.
\end{proposition}
\begin{proof}
    It follows from Theorem~\ref{th.maingeil}.
\end{proof}

We come to one of the main results in this section. This reduction result shows that after removing all common coordinate factors, every remaining extremal polynomial is supported on a decreasing set whose maximal monomials are relatively prime.
\begin{proposition}\label{prop.factorh}
Assume that $\gcd(x^{a_1}y^{b_1},x^{a_2}y^{b_2})=x^ay^b \neq 1$ and take
$f\in Z_S(\mathcal M)$ with $in(f)=\max(\mathcal M)$. There exist subsets $X'\subseteq X$ and $Y'\subseteq Y$ with
$|X'|=a$ and $|Y'|=b$, and a polynomial $h\in R_{S'}$,
$S'=(X\setminus X')\times(Y\setminus Y')$,
such that
\[
f=h\prod_{\alpha\in X'}(x-\alpha)\prod_{\beta\in Y'}(y-\beta).
\]
Moreover, $h\in Z_{S'}(\mathcal M')$, where
$\mathcal M' = \{u\in\Delta(I_{S'}) : ux^ay^b\in\mathcal M\}$.
Furthermore,
$\max(\mathcal M')=\{x^{d_x},y^{d_y}\}$,
where $d_x=\max\{a_1,a_2\}-a$ and $d_y=\max\{b_1,b_2\}-b$.
\end{proposition}

\begin{proof}
Take $f\in Z_S(\mathcal M)$. As 
$in(f)=\max(\mathcal M)=\{x^{a_1}y^{b_1},x^{a_2}y^{b_2}\}$,
we have
\[
\gcd(in(f))=x^ay^b\neq 1.
\]

By Proposition~\ref{prop.linearfactors}, a polynomial in \(Z_S(\mathcal M)\) whose set of possible leading monomials has a nontrivial gcd must have a factor of the form \(x-\alpha\) or \(y-\beta\). Repeating this argument, choose maximal subsets $X'\subseteq X$ and $Y'\subseteq Y$ such that
\[
f=h\prod_{\alpha\in X'}(x-\alpha)\prod_{\beta\in Y'}(y-\beta),
\]
with
$S'=(X\setminus X')\times(Y\setminus Y')$. Define $a'=|X'|$ and $b'=|Y'|$. By Lemma~\ref{lemma.factors},
\[
h\in Z_{S'}(\mathcal M''),
\]
where $\mathcal M''=\{u\in\Delta(I_{S'}) : ux^{a'}y^{b'}\in\mathcal M\}$. Since multiplication by $\prod_{\alpha\in X'}(x-\alpha)\prod_{\beta\in Y'}(y-\beta)$ raises the exponents of every leading monomial by \((a',b')\), we obtain
\[
in(h)=\{x^{a_1-a'}y^{b_1-b'},\;x^{a_2-a'}y^{b_2-b'}\}.
\]
Suppose that $\gcd(in(h))\neq 1$. Then the maximal monomials of \(\mathcal M''\) have a nontrivial common factor. Applying Proposition~\ref{prop.linearfactors} to \(h\), either \(h\notin Z_{S'}(\mathcal M'')\), contradicting Lemma~\ref{lemma.factors}, or \(h\) possesses an additional factor of the form \(x-\alpha\) or \(y-\beta\), contradicting the maximality of \(X'\) and \(Y'\). Therefore $\gcd(in(h))=1$.
Since
\[
\gcd(in(h))=x^{a-a'}y^{b-b'},
\]
it follows that $a'=a$ and $b'=b$. Consequently,
\[
\mathcal M''=\{u\in\Delta(I_{S'}) : ux^ay^b\in\mathcal M\}
=\mathcal M',
\]
and its maximal monomials are $x^{a_1-a}y^{b_1-b}$ and $x^{a_2-a}y^{b_2-b}$. We also have that $x^ay^b=\gcd(x^{a_1}y^{b_1},x^{a_2}y^{b_2})$, so one of the exponents \(a_1-a\), \(a_2-a\) is zero and one of the
exponents \(b_1-b\), \(b_2-b\) is zero. Hence
\[
\max(\mathcal M')=\{x^{d_x},y^{d_y}\},
\]
where $d_x=\max\{a_1,a_2\}-a$ and $d_y=\max\{b_1,b_2\}-b$. This proves the result.
\end{proof}

The previous results imply that, to determine whether noncanonical extremal polynomials exist, i.e., if the following set
\[ Z_S(\mathcal M)\setminus\Pi_S(\mathcal M) \]
is nonempty, it suffices to study decreasing sets $\mathcal M'$ satisfying
$\max(\mathcal M') = \{x^{d_x},y^{d_y}\}$.

\section{Extremal Polynomials with Two Maximal Monomials}
\label{sec:structure}

We continue with the notation from previous sections. In particular, $S=X\times Y\subseteq \F^2$ is a finite Cartesian product with $|X|=n_x$ and $|Y|=n_y$, and $\mathcal M\subseteq \Delta(I_S)$ is a decreasing set. Throughout this section, we assume that 
\[
\max(\mathcal{M})=\{x^{d_x},y^{d_y}\}
\]
and investigate the existence and structure of the elements of
$Z_S(\mathcal M)\setminus \Pi_S(\mathcal M)$.

\subsection{An existence criterion}

We begin by characterizing when noncanonical extremal polynomials exist.

\begin{proposition}\label{prop:noPolyexist}
Assume that $\max(\mathcal{M})=\{x^{d_x},y^{d_y}\}$ and let $h$ be an element in $R_S$ such that $\supp(h)\subseteq\mathcal M$. Then,
\[
h\in Z_S(\mathcal M)\setminus\Pi_S(\mathcal M)
\]
if and only if there exist $f\in\F[x]$ and $g\in\F[y]$, and a subset $\Lambda\subseteq\F$, such that
\begin{alignat*}{3}
h(x,y) &= f(x)+g(y), &\qquad \deg(f) &= d_x, &\qquad \deg(g) &= d_y, \\[1ex]
\frac{n_x}{d_x} &= \frac{n_y}{d_y}\in\mathbb{Z}, &\qquad |\Lambda| &= \frac{n_x}{d_x}, & \qquad \text{and} \qquad f(X) &= -g(Y)=\Lambda.
\end{alignat*}
Moreover, for every $\lambda\in \Lambda$,
$\bigl|f^{-1}(\lambda)\cap X\bigr|=d_x$
and 
$\bigl|g^{-1}(-\lambda)\cap Y\bigr|=d_y$.
\end{proposition}

\begin{proof}
Suppose that $h\in Z_S(\mathcal M)\setminus\Pi_S(\mathcal M)$. By Proposition~\ref{prop.1init}, $in(h)=\{x^{d_x},y^{d_y}\}$. Since
$\max(\mathcal M)=\{x^{d_x},y^{d_y}\}$ and $\mathcal M$ is decreasing, every monomial in $\mathcal M$ is univariate. Therefore, there exist polynomials $f\in\F[x]$ and $g\in\F[y]$ such that
\[
h(x,y)=f(x)+g(y).
\]
The equality $in(h)=\{x^{d_x},y^{d_y}\}$ implies that
\[
\deg(f)=d_x
\qquad\text{and}\qquad
\deg(g)=d_y.
\]
Since $h\in Z_S(\mathcal M)$,
$|Z_S(h)|=z_S(\mathcal M)=\max\{n_yd_x,n_xd_y\}$.
Applying the footprint bound to the two possible leading monomials of
$h$, we obtain
$|Z_S(h)|\leq z_S(d_x,0)=n_yd_x$ and $|Z_S(h)|\leq z_S(0,d_y)=n_xd_y$.
It follows that $n_yd_x=n_xd_y$ and hence
\[
\frac{n_x}{d_x}=\frac{n_y}{d_y}.
\]
Furthermore, $|Z_S(h)|=n_xd_y=n_yd_x$. For each $u\in X$, the polynomial
$h(u,y)=f(u)+g(y)$ has degree $d_y$. Therefore,
\[
\bigl|\{v\in Y:h(u,v)=0\}\bigr|=\bigl|\{v\in Y:g(v)=-f(u)\}\bigr|\leq d_y.
\]
Summing over all $u\in X$ gives
$|Z_S(h)|=\sum_{u\in X}\bigl|\{v\in Y:h(u,v)=0\}\bigr| \leq n_xd_y$.
Since equality holds, every summand must equal $d_y$. Thus, for every
$u\in X$,
\[
\bigl|\{v\in Y:g(v)=-f(u)\}\bigr|=d_y.
\]
In particular, $-f(X)\subseteq g(Y)$. Now set $\Lambda=f(X)$.
For every $\lambda\in \Lambda$, define
\[
X_\lambda=\{u\in X:f(u)=\lambda\}
\quad \text{ and } \quad
Y_\lambda=\{v\in Y:g(v)=-\lambda\}.
\]
Since $\deg(f)=d_x$, $|X_\lambda|\leq d_x$, and the sets $X_\lambda$, for $\lambda\in \Lambda$, partition $X$. Consequently,
$n_x=\sum_{\lambda\in \Lambda}|X_\lambda|\leq |\Lambda|d_x$ and $|\Lambda|\geq\frac{n_x}{d_x}$.

Similarly, the sets $Y_\lambda$, for $\lambda\in \Lambda$, are pairwise disjoint, and each has cardinality $d_y$. Hence
$n_y\geq\sum_{\lambda\in \Lambda}|Y_\lambda|=|\Lambda|d_y$,
which gives $|\Lambda|\leq\frac{n_y}{d_y}=\frac{n_x}{d_x}$.
Therefore,
\[
|\Lambda|=\frac{n_x}{d_x}=\frac{n_y}{d_y} \in \Z.
\]

Since
$n_x=\sum_{\lambda\in \Lambda}|X_\lambda|=|\Lambda|d_x$,
and every fiber $X_\lambda$ has cardinality at most $d_x$, we must have
\[
\bigl|f^{-1}(\lambda)\cap X\bigr|=|X_\lambda|=d_x
\]
for every $\lambda\in \Lambda$. Similarly,
\[
\bigl|g^{-1}(-\lambda)\cap Y\bigr|=|Y_\lambda|=d_y
\]
for every $\lambda\in \Lambda$. Finally, the sets $Y_\lambda$, for $\lambda\in \Lambda$, are pairwise disjoint and their total cardinality is
$|\Lambda|d_y=n_y$. They therefore partition $Y$, and hence
$-g(Y)=\Lambda$.
Thus
\[
f(X)=-g(Y)=\Lambda,
\]
as required.

Conversely, suppose that
\[
h(x,y)=f(x)+g(y),
\qquad
\deg(f)=d_x,
\qquad
\deg(g)=d_y,
\]
and that there exists $\Lambda\subseteq\F$ satisfying the stated
conditions. For each $\lambda\in \Lambda$, let
\[
X_\lambda
=
f^{-1}(\lambda)\cap X
\qquad\text{and}\qquad
Y_\lambda
=
g^{-1}(-\lambda)\cap Y.
\]
Then
$|X_\lambda|=d_x$ and $|Y_\lambda|=d_y$.
Moreover, the zero set of $h$ is the disjoint union
\[
Z_S(h)=
\bigsqcup_{\lambda\in \Lambda}
(X_\lambda\times Y_\lambda).
\]
Therefore,
$|Z_S(h)|=\sum_{\lambda\in \Lambda}|X_\lambda||Y_\lambda|=
|\Lambda|d_xd_y=\frac{n_x}{d_x}d_xd_y=n_xd_y.$
Since $\frac{n_x}{d_x}=\frac{n_y}{d_y}$,
we also have $n_xd_y=n_yd_x=z_S(\mathcal M)$.
Thus, $h\in Z_S(\mathcal M)$. Finally, because
$\deg(f)=d_x$ and $\deg(g)=d_y$, then
$in(h)=\{x^{d_x},y^{d_y}\}$. 
It follows from Proposition~\ref{prop.1init} that
$h\notin\Pi_S(\mathcal M)$.
Therefore,
\[
h\in Z_S(\mathcal M)\setminus\Pi_S(\mathcal M).
\]
This completes the proof.
\end{proof}

The previous proposition shows that the existence of a noncanonical extremal polynomial imposes strong arithmetic restrictions on the grid parameters. In particular, the quotient
\[
\frac{n_x}{d_x}
=
\frac{n_y}{d_y}
\]
must be an integer. Moreover, every noncanonical extremal polynomial induces a partition of the sets $X$ and $Y$ according to $f^{-1}(\lambda)\cap X$ and $g^{-1}(\lambda)\cap Y$.

The next result describes the structure of a polynomial that is constant on several subsets of cardinality equal to its degree. This condition guarantees the existence of noncanonical extremal polynomials.
\begin{lemma}\label{lem:form_h_with_m}
Let $f\in\F[x]$ be a polynomial of degree $k$. Assume that $X_1,\ldots,X_t\subseteq\F$ are $t$ pairwise distinct subsets of cardinality $k$ such that $f(X_i)=\lambda_i$ for $1\leq i \leq t$. The following statements hold: 1. $f^{-1}(\lambda_i)=X_i$ for $1 \leq i \leq t$. 2. For $1 \leq i < j \leq t$, $\lambda_i \neq \lambda_j$ and $X_i \cap X_j = \emptyset$. 3. There exists $\gamma\in\F^*$ such that, for $1\leq i \leq t$,
\[
f(x)=\lambda_i+\gamma\prod_{\alpha\in X_i}(x-\alpha).
\]
Moreover, for $i\neq j$, 
$\gamma=(-1)^k \left(\lambda_j-\lambda_i\right)
\left(\prod_{\alpha\in X_i}\alpha-\prod_{\alpha\in X_j}\alpha\right)^{-1}$, and for $i\neq j$ and $1 \leq  r < k$, we have
\[
\sum_{\substack{ \{\alpha_1,\ldots,\alpha_r\}\subseteq X_i }} \alpha_1\cdots\alpha_r = \sum_{\substack{ \{\alpha_1,\ldots,\alpha_r\}\subseteq X_j }} \alpha_1\cdots\alpha_r.
\]
\end{lemma}
\begin{proof}
The polynomial $f(x)-\lambda_i$ has degree $k$ and $k$ distinct roots belonging to $X_i$. It therefore has no other roots in $\F$. Hence $\{ \alpha \in \F : f(\alpha)=\lambda_i\} = X_i$. For $1 \leq i \leq t$, define
$P_i(x)=\prod_{\alpha\in X_i}(x-\alpha)$. As $f(X_i)=\lambda_i$, every element of $X_i$ is a root of $f(x)-\lambda_i$. Since both $f(x)-\lambda_i$ and $P_i(x)$ have degree $k$, there exists $\gamma_i\in\F^*$ such that
\[
f(x)-\lambda_i=\gamma_iP_i(x).
\]
For $i\neq j$, the previous equation implies that
$\gamma_iP_i(x)-\gamma_jP_j(x)=\lambda_j-\lambda_i$. Since $P_i$ and $P_j$ are monic of degree $k$, comparison of the coefficients of $x^k$ gives $\gamma_i=\gamma_j$. Thus, there exists $\gamma\in\F^*$ such that $\gamma_i=\gamma$
for every $i$. Therefore,
\begin{equation}\label{26.08.17}
f(x)=\lambda_i+\gamma P_i(x)=\lambda_i+\gamma\prod_{\alpha\in X_i}(x-\alpha)    
\end{equation}
for every $i$. If $\lambda_i=\lambda_j$ for some $i\neq j$, Equation~\eqref{26.08.17} and $\gamma\neq0$ imply $P_i(x)=P_j(x)$. Since these polynomials are monic and split into distinct linear factors, we have $X_i=X_j$. Therefore, the elements $\lambda_1,\ldots,\lambda_t$ are pairwise distinct. Now suppose that $X_i\cap X_j\neq\emptyset$ for some $i\neq j$, and pick $\alpha\in X_i\cap X_j$. Then $\lambda_i=f(\alpha)=\lambda_j$, which is a contradiction. Hence the sets $X_1,\ldots, X_t$ are pairwise disjoint.

Finally, write 
\[
P_i(x)=x^k-e_1(X_i)x^{k-1}+e_2(X_i)x^{k-2}-\cdots+(-1)^ke_k(X_i),
\]
where $e_r(X_i)=
\sum_{\substack{ \{\alpha_1,\ldots,\alpha_r\}\subseteq X_i }}
\alpha_1\cdots\alpha_r$. For $i\neq j$, we have
$\gamma\bigl(P_i(x)-P_j(x)\bigr)=\lambda_j-\lambda_i$. Thus, comparing the coefficients of $x^{i}$, we obtain $e_r(X_i)=e_r(X_j)$ for $1\leq r\leq k-1$. The constant terms give
$\gamma(-1)^k
\left(\prod_{\alpha\in X_i}\alpha-\prod_{\alpha\in X_j}\alpha\right)
=\lambda_j-\lambda_i$.
\end{proof}

The previous lemma implies that the sets $f^{-1}(\lambda)\cap X$ and $g^{-1}(\lambda)\cap Y$ associated with distinct $\lambda$'s are pairwise disjoint and satisfy strong algebraic relations. These relations play a central role in classifying extremal polynomials.

We now describe all noncanonical extremal polynomials.
\begin{proposition}\label{prop:form_h_with_m}
Assume that $|X|=td_x$ and $|Y|=td_y$. Let $\mathcal M\subseteq\Delta(I_S)$ be a decreasing set such that $\max(\mathcal M)=\{x^{d_x},y^{d_y}\}$. For every
$h \in Z_S(\mathcal M)\setminus\Pi_S(\mathcal M)$
there exist partitions 
$X=X_1\sqcup\cdots\sqcup X_t$ and $Y=Y_1\sqcup\cdots\sqcup Y_t$ with $|X_i|=d_x$ and $|Y_i|=d_y$ such that
\[
h(x,y)=\rho\bigl(P_i(x)+\mu Q_i(y)\bigr),
\]
where $P_i(x)=\prod_{\alpha\in X_i}(x-\alpha)$ and $Q_i(y)=\prod_{\beta\in Y_i}(y-\beta)$. If $i\neq j$, then
$\mu=(-1)^{d_y-d_x+1}
\left(\prod_{\alpha\in X_i}\alpha-\prod_{\alpha\in X_j}\alpha\right)
\left(\prod_{\beta\in Y_i}\beta-\prod_{\beta\in Y_j}\beta\right)^{-1}$.
\end{proposition}

\begin{proof}
Let $h\in Z_S(\mathcal M)\setminus\Pi_S(\mathcal M)$. By Proposition~\ref{prop:noPolyexist}, there exist $f\in\F[x]$ and $g\in\F[y]$ such that $h(x,y)=f(x)+g(y)$, $\deg(f)=d_x$, and $\deg(g)=d_y$. Moreover, there exists a set $\Lambda=\{\lambda_1,\ldots,\lambda_t\}\subseteq\F$ of cardinality $t$ such that $f(X)=-g(Y)=\Lambda$ and, for every $i\in\{1,\ldots,t\}$, we have $\bigl|f^{-1}(\lambda_i)\cap X\bigr|=d_x$ and $\bigl|g^{-1}(-\lambda_i)\cap Y\bigr|=d_y$.

For each $1 \leq i \leq t$, define $X_i=f^{-1}(\lambda_i)\cap X$ and $Y_i=g^{-1}(-\lambda_i)\cap Y$. The sets $X_i$ and $Y_i$ corresponding to distinct values are pairwise disjoint. Since $|X_i|=d_x$, $|Y_i|=d_y$, $|X|=td_x$, and $|Y|=td_y$, we obtain the partitions $X=X_1\sqcup\cdots\sqcup X_t$ and $Y=Y_1\sqcup\cdots\sqcup Y_t$.

For each $i$, define $P_i(x)=\prod_{\alpha\in X_i}(x-\alpha)$ and $Q_i(y)=\prod_{\beta\in Y_i}(y-\beta)$. Applying Lemma~\ref{lem:form_h_with_m} to $f$ and the sets $X_1,\ldots,X_t$, there exists $\rho\in\F^*$, independent of $i$, such that $f(x)=\lambda_i+\rho P_i(x)$ for every $i$.

Similarly, applying Lemma~\ref{lem:form_h_with_m} to $g$ and the sets $Y_1,\ldots,Y_t$, there exists $\sigma\in\F^*$, independent of $i$, such that $g(y)=-\lambda_i+\sigma Q_i(y)$ for every $i$. Therefore, $$h(x,y)=\rho P_i(x)+\sigma Q_i(y)=\rho\bigl(P_i(x)+\mu Q_i(y)\bigr),$$ where $\mu=\sigma/\rho\in\F^*$. This representation holds for every $1 \leq i \leq t$. For $i\neq j$, Lemma~\ref{lem:form_h_with_m} gives
$\rho=(-1)^{d_x}
\frac{\lambda_j-\lambda_i}
{\prod_{\alpha\in X_i}\alpha-\prod_{\alpha\in X_j}\alpha}
$
and
$\sigma=(-1)^{d_y}
\frac{\lambda_i-\lambda_j}
{\prod_{\beta\in Y_i}\beta-\prod_{\beta\in Y_j}\beta}
$.
Since $\mu=\sigma/\rho$, we obtain the desired expression for $\mu$.
\end{proof}

The previous proposition shows that every polynomial in
\[
Z_S(\mathcal M)\setminus\Pi_S(\mathcal M)
\]
arises from a decomposition $X=X_1\sqcup\cdots\sqcup X_t$ and
$Y=Y_1\sqcup\cdots\sqcup Y_t$ with $|X_i|=d_x$ and $|Y_i|=d_y$.
\begin{remark}
Conversely, it is not difficult to check that every such decomposition satisfying the conditions of Proposition~\ref{prop:form_h_with_m} produces an extremal polynomial.    
\end{remark}
Thus, the classification problem for extremal polynomials reduces to studying suitable families of subsets of $X$ and $Y$.

\subsection{Examples}

We illustrate the previous construction with explicit examples.

\begin{example}\label{ex:cubic}
    Consider the field $\F_{13}$ and take $d_x = d_y =3$ and $t=4$.  
Define
$$
f(x)=x^3, \quad g(y)=-y^3, \quad \text{ and } \quad h(x,y)=x^3-y^3.
$$
The element $2$ is a primitive element of $\F_{13}^*$. Choose
$$
(\lambda_1,\lambda_2,\lambda_3,\lambda_4) = (2^{0\cdot 3},\,2^{1\cdot 3},\,2^{2\cdot 3},\,2^{3\cdot 3}) = (1,8,12,5) \in \F_{13}^4.
$$
For each $i=1,2,3,4$, the equation
$x^3 = a_i$ has exactly three solutions in $\F_{13}$.
Thus, we have:
$$
\begin{aligned}
X_1 = Y_1 &= \{1,3,9\},   \quad &(\lambda_1=1),\\
X_2 = Y_2 &= \{2,6,5\},   \quad &(\lambda_2=8),\\
X_3 = Y_3 &= \{4,12,10\}, \quad &(\lambda_3=12),\\
X_4 = Y_4 &= \{8,11,7\},  \quad &(\lambda_4=5).
\end{aligned}
$$
Set
$X = X=\bigsqcup_{i=1}^4X_i$, $Y = \bigsqcup_{i=1}^4 Y_i.$ So, $|X|=|Y|=12$. In fact,
$X = Y = \F_{13}^*.$
For $x \in X_i$ and $y \in Y_i$, we have $x^3 = \lambda_i$ and $y^3 = \lambda_i$, hence
$
h(x,y)=x^3-y^3 = \lambda_i-\lambda_i = 0,
$
while for $i \neq j$,
$
h(x,y)=\lambda_i-\lambda_j \neq 0.
$
Therefore,
$
Z_{X\times Y}(h) = \bigcup_{i=1}^4 (X_i \times Y_i).
$
\end{example}
\begin{remark}\label{26.08.18}
Let $\ell(y) = \lambda y + \beta$ with $\lambda\neq0, \beta \in \F$ an affine transformation. If a polynomial $h(x,y)$ satisfies Proposition~\ref{prop:form_h_with_m} on the sets $X_1,\ldots,X_t,Y_1,\ldots,Y_t$, then $h(x,\ell(y))$ also satisfies Proposition~\ref{prop:form_h_with_m} on the sets $X_1,\ldots,X_t,Y_1',\ldots,Y_t'$ with $Y_i' = \ell^{-1}{(Y_i)}$.    
\end{remark}

The following example comes from Example~\ref{ex:cubic} and Remark~\ref{26.08.18} with $\ell(y) = y-1$.

\begin{example}
Consider the field $\F_{13}$ and take $d_x = d_y = 3$ and $t = 4$. 
Define
$$
f(x)=x^3, \quad g(y)=-(y-1)^3, \quad \text{ and } \quad h(x,y)=x^3-(y-1)^3.
$$
Thus, we define:
$$
\begin{aligned}
X_1 &= \{1,3,9\},   & Y_1 &= \{2,4,10\} \quad &(a_1=1),\\
X_2 &= \{2,6,5\},   & Y_2 &= \{3,7,6\}  \quad &(a_2=8),\\
X_3 &= \{4,12,10\}, & Y_3 &= \{5,0,11\} \quad &(a_3=12),\\
X_4 &= \{8,11,7\},  & Y_4 &= \{9,12,8\} \quad &(a_4=5).
\end{aligned}
$$
Set
$X = X=\bigsqcup_{i=1}^4X_i$, $Y = \bigsqcup_{i=1}^4 Y_i$. So, $|X|=|Y|=12$ and we have that
$
Z_{X\times Y}(h) = \cup_{i=1}^4 (X_i \times Y_i).
$
\end{example}

These examples show how the zero set decomposes as a union of rectangles
\[
Z_{X\times Y}(h) = \bigcup_{i=1}^t (X_i\times Y_i),
\]
a phenomenon that is characteristic of the extremal polynomials described by Proposition~\ref{prop:form_h_with_m}.

\section{Prouhet--Tarry--Escott Configurations}
\label{sec:PTE}
The relations appearing in Lemma~\ref{lem:form_h_with_m} admit an interpretation in terms of the Prouhet--Tarry--Escott problem. Indeed, the lemma shows that, for every \(i\neq j\),
\[ \deg\left(
\prod_{\alpha\in X_i}(x-\alpha)- \prod_{\alpha\in X_j}(x-\alpha) \right)=0. \]
Equivalently, the elementary symmetric functions of degrees \(1,\ldots,d-1\) agree on \(X_i\) and \(X_j\).

If the characteristic of the field is zero or is sufficiently large relative to the degree, this condition is equivalent to
\[ \sum_{\alpha\in X_i}\alpha^r
= \sum_{\alpha\in X_j}\alpha^r, \qquad 1\leq r\leq d-1 \qquad \text{ and } \qquad 1\leq i < j \leq t.\]
These are the equations appearing in the Prouhet--Tarry--Escott problem. The classical Prouhet--Tarry--Escott problem asks for two disjoint multisets of integers \(A\) and \(B\) satisfying
\[ \sum_{a\in A}a^k = \sum_{b\in B}b^k, \qquad 1\leq k\leq d-1. \]
Such a solution is called ideal if $|A|=|B|=d$.

This is a long-standing problem, and no systematic methods exist for constructing ideal solutions; see \cite{PTE-survey} for a survey of solutions over the integers. 

There are also constructions over the Gaussian integers \cite{PTE-gaussian}. Interestingly, the problem has also appeared in coding theory, in connection with the complexity of Reed--Solomon decoding~\cite{GandikotaGhaziGrigorescu16}. A relaxation of the problem, in which \(A\) and \(B\) are not required to be disjoint, has been studied over \(\mathbb Z_p\) in \cite{PTE-zp}, but a general study over finite fields remains to be developed. Our problem is more restrictive than the usual Prouhet--Tarry--Escott problem: we need a family of pairwise disjoint sets -- not multisets -- satisfying the corresponding relations simultaneously.

\begin{definition}
Let \(\mathcal M\subseteq\Delta(I_S)= \{x^ay^b:0\leq a<n_x,\ 0\leq b<n_y\}\) be a decreasing monomial set. We define
$$ \mathcal T(\mathcal M)
= \left\{ (U, V, h): 
\begin{array}{l} U,V\subseteq\F,\ |U|=n_x,\ |V|=n_y,\\
h\in Z_{U\times V}(\mathcal M) \setminus \Pi_{U\times V}(\mathcal M) 
\end{array} \right\}.$$
\end{definition}

A particularly important case occurs when one of the two maximal monomials has degree one.

\begin{proposition}\label{prop.d1}
Assume there is $t\in\Z_{>0}$ such that $n_x=n_yt$ and
$\max(\mathcal M)=\{x^t,y\}$. Then, $\mathcal T({\mathcal M})\neq\emptyset$ if and only if there exist subsets $X_1,\ldots,X_{n_y}\subseteq\F$, pairwise disjoint each of cardinality $t$, such that
$$
\deg\left(
\prod_{\alpha\in X_i}(x-\alpha)-\prod_{\alpha\in X_j}(x-\alpha)
\right) = 0
$$
for every $1\leq i < j \leq n_y$.
\end{proposition}

\begin{proof}
It follows from Proposition~\ref{prop:form_h_with_m}.
\end{proof}

Lemma~\ref{lem:form_h_with_m} also shows that, to determine all admissible polynomials \(h\), we must first construct the families of subsets satisfying the conditions in Proposition~\ref{prop.d1}. Thus, classifying the relevant extremal polynomials naturally leads to a finite-field version of the Prouhet--Tarry--Escott problem with additional conditions. The following example illustrates how our problem differs from the usual Prouhet--Tarry--Escott problem, even when considered over a finite field.

\begin{example}\label{ex.11}
Let $\mathbb{F}_{11}$ be the field with $11$ elements, and let $\lambda,\mu\in\mathbb{F}_{11}$ be such that $(3\mu-\lambda^2)/7$ is a quadratic residue in $\mathbb{F}_{11}$. According to \cite[Theorem 2.6]{PTE-zp}, there exist three multisubsets $X_1,X_2,X_3$ of $\mathbb{F}_{11}$ satisfying
$$\sum_{\alpha\in X_i}\alpha=\lambda
\qquad \text{ and } \qquad
\sum_{\alpha\in X_i}\alpha^2=\mu.$$
Suppose that $X_1, X_2, X_3$ could be chosen to have no repeated elements and to be, in addition, pairwise disjoint. As
$$\left(\sum_{\alpha\in X_i}\alpha\right)^2
=\sum_{\alpha\in X_i}\alpha^2+
2\sum_{\substack{\alpha,\beta\in X_i\\ \alpha<\beta}}\alpha\beta,
\quad \text{ then } \quad
\sum_{\substack{\alpha,\beta\in X_i\\ \alpha<\beta}}\alpha\beta
=6(\lambda^2-\mu), \quad 1\leq i\leq 3.$$
Therefore,
$$P_i(x)=\prod_{\alpha\in X_i}(x-\alpha)
=x^3-\lambda x^2+6(\lambda^2-\mu)x-\prod_{\alpha\in X_i}\alpha.$$
It follows that $\deg(P_i(x)-P_j(x)) = 0$ whenever $i\neq j$.
Set
$X=X_1\sqcup X_2\sqcup X_3$
and
$Y=\{\prod_{\alpha\in X_i}\alpha:1\leq i\leq3\}$,
and define
$$h(x,y)=x^3-\lambda x^2+6(\lambda^2-\mu)x-y.$$
Then $|X|=9$, $|Y|=3$, and, for every $\alpha\in X_i$,
$h\left(\alpha,\prod_{\gamma\in X_i}\gamma\right)=0$.
Thus, if $\mathcal M=\{1,x,x^2,x^3,y\}$, then
$(X,Y,h)\in\mathcal T(\mathcal M)$.

However, no three \emph{sets} satisfying the stated power-sum conditions can be chosen in $\mathbb{F}_{11}$: of the three multisets given by \cite[Theorem 2.6]{PTE-zp}, exactly one consists of three pairwise distinct elements, the other two containing a repeated element. For instance, for $\lambda=0$ and $\mu=7$ (so that $(3\mu-\lambda^2)/7=3=5^2$, a quadratic residue in $\mathbb{F}_{11}$), the three multisets are
$$\{0,3,8\}, \qquad \{1,5,5\}, \qquad \{6,6,10\},$$
and only the first is a genuine set. Therefore, $\mathcal T(\mathcal M)=\emptyset$ when $n_x=9$, $n_y=3$, and $d=3$.
\end{example}

This example shows that the existence of enough solutions to the Prouhet--Tarry--Escott equations does not by itself guarantee the existence of a noncanonical extremal polynomial. The additional requirement that the relevant subsets be pairwise disjoint can constitute a genuine obstruction.

Thus, the existence and classification of noncanonical extremal polynomials is connected with the construction of families.
\[
X_1,\ldots,X_t\subseteq\F
\]
that are pairwise disjoint and satisfy
\[ \sum_{\alpha\in X_i}\alpha^r
= \sum_{\alpha\in X_j}\alpha^r, \qquad 1\leq r\leq d-1 \qquad \text{ and } \qquad 1\leq i < j \leq t.\]
Equivalently, the associated monic polynomials
\[
P_i(x)=\prod_{\alpha\in X_i}(x-\alpha)
\] must differ only in their constant terms:
$P_i(x)-P_j(x)\in\F.$

In the next section, we focus on the first nontrivial degrees and obtain explicit descriptions and counting formulas for the corresponding extremal polynomials.

\section{Quadratic Extremal Polynomials}
\label{sec:quadratic}

We continue with the notation from previous sections. In particular, $S=X\times Y\subseteq \F^2$ is a finite Cartesian product with $|X|=n_x$ and $|Y|=n_y$, and $\mathcal M\subseteq \Delta(I_S)$ is a decreasing set. In this section, we specialize the results of Section~\ref{sec:structure} to the first nontrivial degrees. More precisely, we study the extremal polynomials arising from the case when
\[
\max(\mathcal{M})=\{x^2,y\}
\qquad\text{or}\qquad
\max(\mathcal{M})=\{x^2,y^2\}.
\]
By Proposition~\ref{prop:noPolyexist}, every nontrivial extremal polynomial is necessarily of the form
\[
h(x,y)=f(x)+g(y),
\]
where the inverse images of $f$ and $g$ determine a partition of the sets $X$ and $Y$. This section makes this description explicit and counts the corresponding admissible triples $(X, Y,h)$ in $\mathcal T({\mathcal M})$.

The behavior depends strongly on the characteristic of the underlying finite field.

\subsection{Fields of characteristic $2$}

Throughout this subsection, we assume that $\F_q$ has characteristic $2$. Since the Frobenius map
$x\longmapsto x^2$ is $\F_2$-linear, quadratic polynomials admit a particularly simple description. This lets us explicitly determine the image sets of the maps defining the extremal polynomials and count all admissible configurations.

We begin with the case
$\max(\mathcal M)=\{x^2,y\}$.
\begin{proposition}\label{prop.char2}
Assume that $\operatorname{char}(\F_q)=2$ and that there exists
$t\in\Z_{>0}$ such that $n_x=2t$ and $n_y=t$. Take
$\mathcal M=\{1,x,x^2,y\}$. Then,
$$|\mathcal T(\mathcal M)|=(q-1)^3q\binom{q/2}{t}.$$
\end{proposition}

\begin{proof}
Let $(X,Y,h)$ a triple in $\mathcal T(\mathcal M)$. By
Propositions~\ref{prop:noPolyexist} and~\ref{prop:form_h_with_m}, we know 
$$h(x,y)=\rho(x^2+\lambda x+\mu y+\nu),$$
where $\rho,\lambda,\mu\in\F_q^*$ and $\nu\in\F_q$. Fix such a polynomial and define $f(x)=x^2+\lambda x$. As $\operatorname{char}(\F_q)=2$, the map $f:\F_q\longrightarrow\F_q$ is $\F_2$-linear. Moreover, $\ker(f)=\{0,\lambda\}$.

We now claim that, for subsets $X,Y\subseteq\F_q$ with $|X|=2t$ and
$|Y|=t$, the triple $(X,Y,h)$ belongs to $\mathcal T(\mathcal M)$ if
and only if there exists a subset
$X'\subseteq\operatorname{Im}(f)$ with $|X'|=t$ such that
$$X=f^{-1}(X')
\text{ and }
Y=\mu^{-1}(X'+\nu).$$
Suppose first that $(X,Y,h)\in\mathcal T(\mathcal M)$. By Proposition~\ref{prop:noPolyexist}, there exists a subset $X'\subseteq\F_q$ with $|X'|=t$ such that $f(X)=-g(Y)=-(\mu Y+\nu)=X'$ and $|f^{-1}(\xi)\cap X|=2$ for $\xi\in X'$. In particular, $X'=f(X)\subseteq\operatorname{Im}(f)$. As $\ker(f)=\{0,\lambda\}$, the equality $|f^{-1}(\xi)\cap X|=2$ implies $f^{-1}(\xi)\subseteq X$ for every $\xi\in X'$. Hence $f^{-1}(X')\subseteq X$. As $f(X)=X'$, $X \subseteq f^{-1}(f(X)) = f^{-1}(X')$. Therefore, $X=f^{-1}(X')$. Moreover, since $\operatorname{char}(\F_q)=2$ and $\mu\neq0$, it follows that $Y=\mu^{-1}(X'+\nu)$. 

Conversely, suppose that there exists a subset $X'\subseteq\operatorname{Im}(f)$ with $|X'|=t$ such that $X=f^{-1}(X')$ and $Y=\mu^{-1}(X'+\nu)$. For every $\xi\in X'$, set $\beta_\xi=\mu^{-1}(\xi+\nu)$. Then $\beta_\xi\in Y$, and $h(\alpha,\beta_\xi)=0$ if and only if $f(\alpha)=\xi$. Thus, $Z_{X\times Y}(h) = \bigsqcup_{\xi\in X'} \left(f^{-1}(\xi)\times\{\beta_\xi\}\right)$. Since $|f^{-1}(\xi)| = 2$, we obtain $|Z_{X\times Y}(h)|=2|X'|=2t=n_x =z_{X\times Y}(\mathcal M)$. Thus $h\in Z_{X\times Y}(\mathcal M)$. Finally, $in(h)=\{x^2,y\}$, so $|in(h)\cap\max(\mathcal M)|=2$. By Proposition~\ref{prop.1init}, $h\notin\Pi_{X\times Y}(\mathcal M)$. Therefore, $(X,Y,h)\in\mathcal T(\mathcal M)$.

We now need to count. There are $\binom{q/2}{t}$ choices for $X'$. There are $q-1$ choices for each of $\rho$, $\lambda$, and $\mu$, and there are $q$ choices for $\nu$. Since the representation $h(x,y)=\rho(x^2+\lambda x+\mu y+\nu)$ is unique, distinct choices of these parameters give distinct polynomials. Therefore,
$|\mathcal T(\mathcal M)|=(q-1)^3q\binom{q/2}{t}$,
as claimed.
\end{proof}

For the case $\max(\mathcal M)=\{x^2,y^2\}$, the essential ingredient is the size of the intersection of the image sets of two quadratic maps.

\begin{lemma}\label{lem:imageIntersection} Assume that $\operatorname{char}(\F_q)=2$. Let $f(x)=x^2+\lambda x$ and $g(y)=\mu y^2+\nu y+\rho$, where $\lambda,\mu,\nu\in\F_q^*$ and $\rho\in\F_q$. Then,
$$|\operatorname{Im}(f)\cap\mathrm{Im}(g)|=
\begin{cases}
\frac{q}{2}& \text{ if } \; \lambda^2=\frac{\nu^2}{\mu} \text{ and }  \rho\in\operatorname{Im}(f),\\
0& \text{ if } \; \lambda^2=\frac{\nu^2}{\mu}  \text{ and } \rho\notin\operatorname{Im}(f),\\
\frac{q}{4}&\text{ otherwise. }
\end{cases}$$
\end{lemma}
\begin{proof}
Let $g_0(y)=g(y)-\rho=\mu y^2+\nu y$. Both $f$ and $g_0$ are
$\F_2$-linear maps. Moreover, as
$\ker(f)=\{0,\lambda\}$ and
$\ker(g_0)=\{0,\nu/\mu\}$, we get
$|\operatorname{Im}(f)|=|\operatorname{Im}(g_0)|=q/2$.

Let $\operatorname{Tr}:\F_q\longrightarrow\F_2$ be the trace. As
$\{\alpha^2+\alpha : \alpha \in \F_q \} = \{\beta \in \F_q : \operatorname{Tr}(\beta) =0\}$,
\begin{align*}
\operatorname{Im}(f)
=&\{\alpha^2+\lambda\alpha : \alpha \in \F_q\}
=\{\lambda^2((\alpha/\lambda)^2+\alpha/\lambda) : \alpha \in \F_q \}\\
=&\{\lambda^2(\alpha^2+\alpha) : \alpha \in \F_q \}
=\{\lambda^2 \beta\in\F_q:\operatorname{Tr}(\beta)=0\}\\
=&\{u\in\F_q:\operatorname{Tr}(u/\lambda^2)=0\}.
\end{align*}
Similarly, as
$g_0(y)=(\nu^2/\mu)((\mu y/\nu)^2+\mu y/\nu)$,
\begin{align*}
\operatorname{Im}(g_0)
&=\{u\in\F_q:\operatorname{Tr}(\mu u/\nu^2)=0\}
\qquad \text{ and } \qquad
\operatorname{Im}(g) = \operatorname{Im}(g_0) + \rho.
\end{align*}

Define the $\F_2$-linear maps $F_1 \colon u\mapsto\operatorname{Tr}(u/\lambda^2)$ and $F_2 \colon u\mapsto\operatorname{Tr}(\mu u/\nu^2)$. Note that $\operatorname{Im}(f) = \ker(F_1)$, $\operatorname{Im}(g_0) = \ker(F_2)$, and $\operatorname{Im}(g)$ is an equivalent class of $\ker(F_2)$.

If $\lambda^2=\nu^2/\mu$, then $\ker(F_1) = \ker(F_2)$. If $\rho\in\operatorname{Im}(f)$, then 
$\operatorname{Im}(g) = \ker(F_2)$ and 
$|\operatorname{Im}(f)\cap\operatorname{Im}(g)|=q/2$. 
If $\rho\notin\operatorname{Im}(f)$, then
$\operatorname{Im}(g)$ is an equivalent class of $\ker(F_2)$, therefore
$|\operatorname{Im}(f)\cap\operatorname{Im}(g)|=0$.

Assume now that $\lambda^2\neq\nu^2/\mu$. Then $\ker(F_1)$ and $\ker(F_2)$ are distinct codimension-one $\F_2$-subspaces of $\F_q$. Their intersection has codimension two, and hence
$|\operatorname{Im}(f)\cap\operatorname{Im}(g_0)|=q/4$. As $\ker(F_1)$ is not contained in $\ker(F_2)$, each equivalent class of $\ker(F_2)$ meets $\ker(F_1)$ in the same number of elements.
\end{proof}

The lemma leads to an explicit enumeration of all extremal triples.
\begin{theorem}\label{thm.char2}
Assume that $\operatorname{char}(\F_q)=2$ and that there exists
$t \leq q/2 \in\Z_{>0}$ such that $n_x=n_y=2t$. Take
$\mathcal M=\{1,x,x^2,y,y^2\}$. Then,
\[
|\mathcal T(\mathcal M)|=
\frac{(q-1)^3q}{2}
\left[\binom{q/2}{t}+2(q-2)\binom{q/4}{t}\right].
\]
\end{theorem}

\begin{proof}
Let $h(x,y)=\alpha(f(x)+g(y))$, with
$f(x)=x^2+\beta x$ and
$g(y)=\lambda y^2+\mu y+\rho$, where
$X,Y\subseteq\F_q$ have cardinality $2t$.

Let $\xi\in\operatorname{Im}(f)\cap\operatorname{Im}(g)$. If
$\beta\mu\neq0$, then
$X_\xi=f^{-1}(\xi)$ and
$Y_\xi=g^{-1}(\xi)$ are subsets of cardinality $2$ such that
$|Z_{X_\xi\times Y_\xi}(h)|=4$. By
Proposition~\ref{prop:noPolyexist}, the last condition implies that
$|Z_{X\times Y}(h)|
=
z_{X\times Y}(\mathcal M)
=
4t$
if and only if $\beta\mu\neq0$ and there exists
$X'\subseteq\operatorname{Im}(f)\cap\operatorname{Im}(g)$, with
$|X'|=t$, such that
$X=f^{-1}(X')$ and $Y=g^{-1}(X')$.

Therefore, for each
$(X,Y,h)\in\mathcal T(\mathcal M)$, there are unique
$\alpha,\beta,\lambda,\mu,\rho\in\F_q^*$, $\rho\in\F_q$, and
$X'\subseteq\operatorname{Im}(f)\cap\operatorname{Im}(g)$, with
$|X'|=t$, where
$f(x)=x^2+\beta x$,
$g(y)=\lambda y^2+\mu y+\rho$, and
$h=\alpha(f+g)$.

Therefore, the possible choices for
$\alpha,\beta,\lambda,\mu,\rho$ and $X'$ are as follows, according to
Lemma~\ref{lem:imageIntersection}:

\begin{itemize}
    \item There are $(q-1)^3$ choices for
    $\alpha,\beta,\mu\in\F_q^*$.

    \item If $\lambda$ is such that
    $\mu^2=\lambda\beta^2$, we have to choose
    $\rho\in\operatorname{Im}(f)$; otherwise,
    $\operatorname{Im}(f)\cap\operatorname{Im}(g)$ would be empty.
    After choosing $\rho$, we may choose any
    $X'\subseteq\operatorname{Im}(f)\cap\operatorname{Im}(g)$ of
    cardinality $t$. In total, there are
    $\underbrace{q/2}_{\rho}
    \binom{q/2}{t}$
    possibilities.

    \item If we choose any other $\lambda$, we may choose $\rho$
    freely and choose $X'$ from an intersection of cardinality
    $q/4$. Thus, there are
    $\underbrace{(q-2)}_{\lambda}
    \underbrace{q}_{\rho}
    \binom{q/4}{t}$ 
    possibilities.
\end{itemize}

Thus, the total number of possibilities is
$(q-1)^3
\left[
\frac{q}{2}\binom{q/2}{t}+(q-2)q\binom{q/4}{t}
\right]$.
\end{proof}

The following examples illustrate the previous theorem and provide alternative descriptions of the corresponding extremal polynomials.

\begin{example}\rm
    Let $\F_q$ a finite field of characteristic $2$ and consider the set $\mathcal{M}=\{1,x,y,x^2,y^2\}$ and $S=\F_q^2$. By Theorem \ref{thm:minRM}, we know that any $h\in Z_S(\mathcal{M})$ is of the form 
    $$(\alpha x+\beta y+\gamma_1)(\alpha x+\beta y+\gamma_2).$$

    But the opposite is also true due to the characteristic; in fact, for any $\alpha,\beta,\gamma_1,\gamma_2\in\F_q$, with $\alpha,\beta\neq 0$ and $\gamma_1\neq\gamma_2$, we have
    $$(\alpha x+\beta y+\gamma_1)(\alpha x+\beta y+\gamma_2)=\alpha^2x^2+\beta^2y^2+(\gamma_1+\gamma_2)(\alpha x+\beta y)+\gamma_1\gamma_2\in Z_S(\mathcal{M})$$

    The total number of polynomials with the previous form and $\alpha\beta\neq 0$ is
    $$(q-1)^2{q\choose 2}=\frac{(q-1)^3q}{2},$$

    \noindent which coincides with the number given by Theorem \ref{thm.char2}.
\end{example}

\begin{example}
Let us verify the formula in Theorem~\ref{thm.char2} for $t=2$. Take
\[
X=\{\alpha_1,\alpha_2,\alpha_3,\alpha_4\}
\qquad \text{ and } \qquad
Y=\{\beta_1,\beta_2,\beta_3,\beta_4\}
\]
and suppose that $h=\rho(f(x)+g(y))$ has $8$ zeros in $X\times Y$, where
\[
f(x)=x^2+\mu x \qquad \text{ and } \qquad g(y)=\lambda_1(y^2+\lambda_2y+\lambda_3).
\]
As mentioned above, this implies that, up to permutation, there are
$\gamma_1,\gamma_2\in\F_q$, with $\gamma_1\neq\gamma_2$, such that
\begin{align*}
f(x)&=(x+\alpha_1)(x+\alpha_2)+\gamma_1
=(x+\alpha_3)(x+\alpha_4)+\gamma_2 \qquad \text{ and }\\
g(y)&=\lambda_1(y+\beta_1)(y+\beta_2)+\gamma_1
=\lambda_1(y+\beta_3)(y+\beta_4)+\gamma_2.
\end{align*}
These two identities imply that $\sum_{i=1}^4 \alpha_i = \sum_{i=1}^4 \beta_i = 0$
and
$\gamma_1+\gamma_2
=\alpha_1\alpha_2+\alpha_3\alpha_4
=\lambda_1(\beta_1\beta_2+\beta_3\beta_4)$.
Therefore,
$\lambda_1=
\left(\alpha_1\alpha_2+\alpha_3\alpha_4\right)
\left(\beta_1\beta_2+\beta_3\beta_4\right)^{-1}$.

Conversely, choose
$\rho,\alpha_1,\alpha_2,\alpha_3,\beta_1,\beta_2,\beta_3\in\F_q$,
with $\rho\neq0$, such that
$\alpha_i\neq\alpha_j$ and $\beta_i\neq\beta_j$ whenever
$1\leq i<j\leq3$. Define
$\alpha_4=\alpha_1+\alpha_2+\alpha_3$,
$\beta_4=\beta_1+\beta_2+\beta_3$,
and
$\lambda_1=
\left(\alpha_1\alpha_2+\alpha_3\alpha_4\right)
\left(\beta_1\beta_2+\beta_3\beta_4\right)^{-1}$.
Then,
\[
h(x,y)
=\rho\bigl((x+\alpha_1)(x+\alpha_2)
+\lambda_1(y+\beta_1)(y+\beta_2)\bigr)
\]
is a polynomial with $8$ zeros in
$\{\alpha_1,\alpha_2,\alpha_3,\alpha_4\}
\times \{\beta_1,\beta_2,\beta_3,\beta_4\}$.
Thus, the number of elements in $\mathcal T(\mathcal M)$ is obtained as follows:

\begin{itemize}
    \item The sets $\{\alpha_i\}_{i=1}^3$ and $\{\beta_i\}_{i=1}^3$
    and $\rho\in\F_q^*$ give ${\binom{q}{3}^2}{(q-1)}$
    possibilities.

    \item There are
    $\binom{4}{2}^2=36$
    ways to choose two elements of
    $\{\alpha_1,\alpha_2,\alpha_3,\alpha_4\}$ and two elements of
    $\{\beta_1,\beta_2,\beta_3,\beta_4\}$ to write
    $f(x)=(x+\alpha_1)(x+\alpha_2)+\gamma_1$
    and
    $g(y)=\lambda_1(y+\beta_1)(y+\beta_2)+\gamma_1$.

    \item The choice of the pair
    $\{(\alpha_1,\alpha_2),(\beta_1,\beta_2)\}$
    is equivalent to the choice of the complementary pair
    $\{(\alpha_3,\alpha_4),(\beta_3,\beta_4)\}$.
    Therefore, we must divide the previous factor by $2$.

    \item Finally, the ordered choices
    $(\alpha_1,\alpha_2,\alpha_3,\alpha_4)$,
    $(\alpha_1,\alpha_2,\alpha_4,\alpha_3)$,
    $(\alpha_2,\alpha_1,\alpha_3,\alpha_4)$, and
    $(\alpha_2,\alpha_1,\alpha_4,\alpha_3)$
    produce the same polynomial. The analogous statement holds for
    the elements $\beta_i$. Thus, we must divide the total number of
    choices by $16$.
\end{itemize}
Multiplying these factors, we obtain
$|\mathcal T(\mathcal M)|=
\frac{36}{32}(q-1)\binom{q}{3}^2$.

If we use Theorem~\ref{thm.char2}, we obtain
\begin{align*}
|\mathcal T(\mathcal M)|
&=
\frac{(q-1)^3q}{2}
\left[
\binom{q/2}{2}+2(q-2)\binom{q/4}{2}
\right]\\
&=
\frac{(q-1)^3q^2}{8}
\left[\left(\frac q2-1\right)+(q-2)\left(\frac q4-1\right)
\right]\\
&=
\frac{(q-1)^3q^2}{8}
\frac{q-2}{2}
\left(\frac q2-1\right)
=
\frac{(q-1)^3q^2(q-2)^2}{32},
\end{align*}
which agrees with the explicit count above.
\end{example}

\subsection{Fields of odd characteristic}

We now assume that $q$ is odd. In this situation, quadratic
polynomials are naturally related to the set
\[
\operatorname{SQ}_q=
\{\lambda^2:\lambda\in\F_q^*\}
\]
of nonzero squares. As in the characteristic-two case,
the classification problem reduces to counting suitable intersections
of image sets.

We first study the case
$\max(\mathcal M)=\{x^2,y\}$.

\begin{proposition}\label{prop.charneq2}
Assume that $\operatorname{char}(\F_q)$ is odd and that there exists
$t\in\Z_{>0}$ such that $n_x=2t$ and $n_y=t$. Take
$\mathcal M=\{1,x,x^2,y\}$. Then,
\[
|\mathcal T(\mathcal M)|=
(q-1)^2q^2\binom{(q-1)/2}{t}.
\]
\end{proposition}
\begin{proof}
Let
$h=\alpha((x-\beta)^2+\lambda y+\mu)$,
where $\alpha,\lambda\neq0$. Take $f(x)=(x-\beta)^2$ and $g(y)=\lambda y+\mu$. For $X,Y\subseteq\F_q$, we have $(X,Y,h)\in\mathcal T(\mathcal M)$ if and only if there exists $X'\subseteq \operatorname{SQ}_q$, with $|X'|=t$, such that $X=f^{-1}(X')$ and $Y=\lambda^{-1}(X'-\mu)$. Thus, we only need to count the possible choices of $X'$. In total, we have
\[
|\mathcal T(\mathcal M)|=
\underbrace{(q-1)^2}_{\alpha,\lambda}
\underbrace{q^2}_{\beta,\mu}
\binom{(q-1)/2}{t},
\]
which completes the proof.
\end{proof}

To treat the case $\max(\mathcal M)=\{x^2,y^2\}$, we introduce the sets 
\[
\operatorname{SQ}_{\alpha,\beta}=\{u\in\operatorname{SQ}_q:\alpha u+\beta \in \operatorname{SQ}_q\},
\]
whose cardinalities determine the number of admissible decompositions. The next lemma provides a closed formula for their size.

\begin{lemma}\label{lem:SQ_a_b_general_form}
Assume that $\operatorname{char}(\F_q)$ is odd and take
$(\alpha,\beta)\in\F_q^*\times\F_q^*$. Then,
$$
|\operatorname{SQ}_{\alpha,\beta}|
=
\frac{
q-2-\chi(\alpha)-\chi(\beta)
-\chi(-1)\chi(\alpha)\chi(\beta)
}{4}.
$$
\end{lemma}

\begin{proof}
We observe that for $u\in\F_q^*$, the indicator function of
$\operatorname{SQ}_q$ is $\frac{1+\chi(u)}{2}$. Then
$$
|\operatorname{SQ}_{\alpha,\beta}|
=
\sum_{\substack{u\in\F_q^*\\ \alpha u+\beta\neq0}}
\frac{1+\chi(u)}{2}
\cdot
\frac{1+\chi(\alpha u+\beta)}{2}.
$$
We now extend the sum over all $u\in\F_q^*$ and subtract the
contribution at $u_0=-\beta/\alpha$:
$$
|\operatorname{SQ}_{\alpha,\beta}|
=
\sum_{u\in\F_q^*}
\frac{1+\chi(u)}{2}
\cdot
\frac{1+\chi(\alpha u+\beta)}{2}
-
\frac{1+\chi(u_0)}{2}
\cdot
\frac{1+\chi(0)}{2}.
$$
Since $\chi(0)=0$, this becomes
\begin{equation}\label{eq:square_ab_chi}
|\operatorname{SQ}_{\alpha,\beta}|
=
\sum_{u\in\F_q^*}
\frac{1+\chi(u)}{2}
\cdot
\frac{1+\chi(\alpha u+\beta)}{2}
-
\frac{1+\chi(u_0)}{4}.
\end{equation}
We now expand the sum:
$$
\sum_{u\in\F_q^*}
\frac{1+\chi(u)}{2}
\cdot
\frac{1+\chi(\alpha u+\beta)}{2}
=
\frac{1}{4}
\sum_{u\in\F_q^*}
\bigl(
1+\chi(u)+\chi(\alpha u+\beta)
+\chi(u)\chi(\alpha u+\beta)
\bigr).
$$
We evaluate each term separately. Note that
$\sum_{u\in\F_q^*}1=q-1$ and
$\sum_{u\in\F_q^*}\chi(u)=0$.
Consider $\sum_{u\in\F_q^*}\chi(\alpha u+\beta)$.

The map $u\mapsto v=\alpha u+\beta$ is a bijection from
$\F_q^*$ onto $\F_q\setminus\{\beta\}$. Hence,
$$
\sum_{u\in\F_q^*}\chi(\alpha u+\beta)
=
\sum_{v\in\F_q\setminus\{\beta\}}\chi(v)
=
-\chi(\beta).
$$
Using that $\chi(x)\chi(y)=\chi(x^{-1}y)$, we get
$$
\sum_{u\in\F_q^*}\chi(u)\chi(\alpha u+\beta)
=
\sum_{u\in\F_q^*}\chi(\alpha+\beta u^{-1})
=
\sum_{\substack{u\in\F_q\\u\neq\alpha}}\chi(u)
=
-\chi(\alpha).
$$

Combining all contributions, we obtain
$$
\sum_{u\in\F_q^*}
\frac{1+\chi(u)}{2}
\cdot
\frac{1+\chi(\alpha u+\beta)}{2}
=
\frac{1}{4}
\bigl(
(q-1)-\chi(\beta)-\chi(\alpha)
\bigr).
$$
Since $u_0=-\beta/\alpha$, from
\eqref{eq:square_ab_chi} we obtain
$$
|\operatorname{SQ}_{\alpha,\beta}|
=
\frac{
q-2-\chi(\alpha)-\chi(\beta)-\chi(-\beta/\alpha)
}{4}.
$$
Finally,
$\chi(-\beta/\alpha)
=
\chi(-1)\chi(\alpha)\chi(\beta)$,
which completes the proof.
\end{proof}

Summing these cardinalities yields the following counting formula.

\begin{corollary}\label{cor:SQ_a_b_sum}
Assume that $\operatorname{char}(\F_q)$
is odd and take $1\leq t\leq \frac{q-1}{2}$. Then, 
$$
\sum_{\alpha\in\F_q^*}\sum_{\beta\in\F_q}
\binom{|\operatorname{SQ}_{\alpha,\beta}|}{t}
=
\frac{q-1}{2}\binom{\frac{q-1}{2}}{t}
+
\begin{cases}
\displaystyle
\frac{(q-1)^2}{4}
\left(
\binom{\frac{q-5}{4}}{t}
+
3\binom{\frac{q-1}{4}}{t}
\right),\\
\hfill \text{if }q\equiv1\pmod 4,\\[10pt]
\displaystyle
\frac{(q-1)^2}{4}
\left(
3\binom{\frac{q-3}{4}}{t}
+
\binom{\frac{q+1}{4}}{t}
\right),\\
\hfill \text{if }q\equiv3\pmod 4.
\end{cases}
$$
\end{corollary}

\begin{proof}
We can write
$$
\sum_{\alpha \in \F_q^*}\sum_{\beta\in\F_q}
\binom{|\operatorname{SQ}_{\alpha,\beta}|}{t}
=
\sum_{\alpha \in \F_q^*}
\binom{|\operatorname{SQ}_{\alpha,0}|}{t}
+
\sum_{\alpha \in \F_q^*}\sum_{\beta\in\F_q^*}
\binom{|\operatorname{SQ}_{\alpha,\beta}|}{t}.
$$
Since $|\operatorname{SQ}_{\alpha,0}|=(q-1)/2$ if
$\chi(\alpha)=1$ and $|\operatorname{SQ}_{\alpha,0}|=0$ if
$\chi(\alpha)=-1$, and there are $(q-1)/2$ choices of
$\alpha\in\F_q^*$ with $\chi(\alpha)=1$, we have that 
$$
\sum_{\alpha\in\F_q^*}
\binom{|\operatorname{SQ}_{\alpha,0}|}{t} 
=
\frac{q-1}{2}\binom{\frac{q-1}{2}}{t}.
$$
Next, consider $\beta\neq0$. We now want to compute
$$
\sum_{\alpha \in \F_q^*}\sum_{\beta\in\F_q^*}
\binom{|\operatorname{SQ}_{\alpha,\beta}|}{t}.
$$
There are $q-1$ choices for $\beta$ and $q-1$ choices for $\alpha$,
and for each we have
$\chi(\alpha)\in\{\pm1\}$ and $\chi(\beta)\in\{\pm1\}$. Since
$$
\bigl|\{u\in\F_q^*:\chi(u)=1\}\bigr|
=
\bigl|\{u\in\F_q^*:\chi(u)=-1\}\bigr|
=
\frac{q-1}{2},
$$
each combination $(\chi(\alpha),\chi(\beta))$ occurs
$(q-1)^2/4$ times. In these terms, from
Lemma~\ref{lem:SQ_a_b_general_form}, we have
$$
|\operatorname{SQ}_{\alpha,\beta}|
=
\frac{
q-2-\chi(\alpha)-\chi(\beta)
-\chi(-1)\chi(\alpha)\chi(\beta)
}{4}.
$$
Note that $\chi(-1)=1$ if $q\equiv1\pmod4$ and
$\chi(-1)=-1$ if $q\equiv3\pmod4$. Thus, in the case
$q\equiv1\pmod4$, the four cases
$(\chi(\alpha),\chi(\beta))$ give 
$$
|\operatorname{SQ}_{\alpha,\beta}|
=
\frac{q-5}{4},\
\frac{q-1}{4},\
\frac{q-1}{4},\
\frac{q-1}{4},
$$
for
$(\chi(\alpha),\chi(\beta))
=
(+1,+1),(+1,-1),(-1,+1),(-1,-1)$,
respectively. Hence, their binomial contributions are
$\binom{\frac{q-5}{4}}{t}$ once and
$\binom{\frac{q-1}{4}}{t}$ three times.

Similarly, in the case $q\equiv3\pmod4$, the four cases
$(\chi(\alpha),\chi(\beta))$ give 
$$
|\operatorname{SQ}_{\alpha,\beta}|
=
\frac{q-3}{4},\
\frac{q-3}{4},\
\frac{q-3}{4},\
\frac{q+1}{4},
$$
respectively. Hence, their binomial contributions are
$\binom{\frac{q+1}{4}}{t}$ once and
$\binom{\frac{q-3}{4}}{t}$ three times. Thus,
$$
\sum_{\alpha \in \F_q^*}\sum_{\beta\in\F_q^*}
\binom{|\operatorname{SQ}_{\alpha,\beta}|}{t}
=
\frac{(q-1)^2}{4}
\begin{cases}
\displaystyle \binom{\frac{q-5}{4}}{t}
+3\binom{\frac{q-1}{4}}{t}
\; \text{ if }\; q \equiv 1 \pmod{4}, \\[8pt]
\displaystyle 3\binom{\frac{q-3}{4}}{t}
+\binom{\frac{q+1}{4}}{t}
\; \text{ if }\; q \equiv 3 \pmod{4}. 
\end{cases}
$$
This completes the proof.
\end{proof}

Combining the previous results, we obtain the complete enumeration of extremal triples in odd characteristic.

\begin{theorem}\label{thm.charodd}
Assume that $\operatorname{char}(\F_q)$ is odd and that there is $1\leq a\leq q-1$ such that $t=\frac{q-a}{2}\in\Z$ and $n_x=n_y=2t$. Take $\mathcal M=\{1,x,x^2,y,y^2\}$. Then,
$$
|\mathcal T(\mathcal M)|=
\frac{q^2(q-1)^2}{2}\binom{\frac{q-1}{2}}{t}
+\frac{q^2(q-1)^3}{4}
\begin{cases}
\displaystyle
\left(
\binom{\frac{q-5}{4}}{t}+3\binom{\frac{q-1}{4}}{t}
\right),
\text{if }q\equiv1\pmod 4,\\
\displaystyle
\left(
3\binom{\frac{q-3}{4}}{t}+\binom{\frac{q+1}{4}}{t}
\right),
\hfill \text{if }q\equiv3\pmod 4.
\end{cases}
$$
\end{theorem}

\begin{proof}
Every polynomial of the form
$h(x,y)=f(x)+g(y)\in\F_q[x,y]$ with
$\deg(f)=\deg(g)=2$ can be uniquely written as
$$
h(x,y)
=
\rho\bigl[(x-\gamma)^2
-\alpha(y-\delta)^2-\beta\bigr],
$$
with
$\alpha,\beta,\gamma,\delta,\rho\in\F_q$ and
$\alpha,\rho\neq0$. We denote
$$
f_1(x):=(x-\gamma)^2
\qquad\text{and}\qquad
g_1(y):=\alpha(y-\delta)^2+\beta.
$$

For every $X,Y\subseteq\F_q$ with $|X|=|Y|=2t$, by
Proposition~\ref{prop:noPolyexist}, we know that
$|Z_{X\times Y}(h)|=4t$ if and only if there exist
$\lambda_1,\ldots,\lambda_t\in\F_q$ such that $X$ is a disjoint union
of sets $X_1,\ldots,X_t$ of size two, $Y$ is a disjoint union of sets
$Y_1,\ldots,Y_t$ of size two, and
$f_1(X_i)=\lambda_i$
and
$g_1(Y_i)=-\lambda_i$
for every $i\in\{1,\ldots,t\}$. Therefore,
$
\{\lambda_1,\ldots,\lambda_t\}
\subseteq \operatorname{SQ}_{\alpha,\beta}.
$

Conversely, for every
$
\{\lambda_1,\ldots,\lambda_t\}
\subseteq \operatorname{SQ}_{\alpha,\beta},
$
we can find pairwise disjoint sets $X_1,\ldots,X_t$ of size two and
pairwise disjoint sets $Y_1,\ldots,Y_t$ of size two such that
$f_1(X_i)=\lambda_i$
and
$g_1(Y_i)=-\lambda_i$
for every $i\in\{1,\ldots,t\}$. Therefore, for
$$
X:=\bigcup_{i=1}^tX_i
\qquad\text{and}\qquad
Y:=\bigcup_{i=1}^tY_i,
$$
we have
$
h\in Z_{X\times Y}(\mathcal M)
\setminus \Pi_{X\times Y}(\mathcal M).
$

We have proved that, for fixed
$(\alpha,\beta)\in\F_q^*\times\F_q$ and for every subset of size $t$
of $\operatorname{SQ}_{\alpha,\beta}$, there exist unique
$X,Y\subseteq\F_q$ of size $2t$ such that
$|Z_{X\times Y}(h)|=4t.$
Finally, we count all possible triples $(X, Y,h)$. Since the elements
$\gamma,\delta\in\F_q$ and $\rho\in\F_q^*$ can be freely chosen, we
get
$$
|\mathcal T(\mathcal M)|
=
q^2(q-1)
\sum_{\alpha\in\F_q^*}
\sum_{\beta\in\F_q}
\binom{|\operatorname{SQ}_{\alpha,\beta}|}{t}.
$$
By Corollary~\ref{cor:SQ_a_b_sum}, we obtain the desired expression.
\end{proof}

Theorems~\ref{thm.char2} and~\ref{thm.charodd} completely describe the quadratic case. In particular, they determine exactly when noncanonical extremal polynomials exist and how many admissible triples $(X, Y,h)$ realize the maximum number of zeros.

In the next section, we apply these structural and enumerative results to monomial decreasing evaluation codes and determine the corresponding minimum-weight codewords.

\section{Applications to Monomial Decreasing Codes}
\label{sec:codes}

In this section, we apply the results obtained in the previous sections to monomial decreasing codes. The connection is immediate: minimum-weight codewords correspond to polynomials attaining the largest possible number of zeros on the evaluation grid. Consequently, the structural and enumerative results established for the sets $Z_S(\mathcal M)$ translate into descriptions and counting formulas for minimum-weight codewords.

\subsection{Evaluation codes on Cartesian grids}
Let $S=X\times Y$ be a Cartesian set, where $X$ and $Y$ are subsets of a finite field $\F_q$ with $|X|=n_x$ and $|Y|=n_y$, and let $\mathcal{M}$ be a subset of $\Delta(I_S)$. Write $S=\{P_1,\ldots,P_n\}$, where $n=n_xn_y$. We consider the evaluation map
\[
\ev_S \colon \F_q[x,y] \longrightarrow \F_q^n,
\qquad
f \longmapsto (f(P_1),\ldots,f(P_n)).
\]
When $S=\F_q^2$, we write $\ev=\ev_{\F_q^2}$.
\begin{definition}
The evaluation code associated with $S$ and $\mathcal{M}$ is
\[
\mathcal C(S,\mathcal M) =
\ev_S(\Span_{\F_q}(\mathcal{M})).
\]
\end{definition}
If $\mathcal M \subseteq \Delta(I_S)$ is decreasing, then $\mathcal C(S,\mathcal M)$ is called a monomial decreasing code. 

We define the following footprint values on $S$.
\begin{itemize}
\item The footprint of $m=x^{a}y^{b}\in \Delta(I_S)$ is $\FB_S(m)=(n_x-a)(n_y-b)$.
\item The footprint of $f\in \Span_{\F_q}(\mathcal{M})$ is 
$\FB_S(f)=\min\{\FB_S(m):m\in in(f)\}$.
\item The footprint of $\mathcal C(S,\mathcal M)$ is
$\FB_S(\mathcal C(S,\mathcal M))=
\min\{\FB_S(m):m\in\mathcal M\}$.
\end{itemize}

The footprint values introduced above are naturally related to the quantities studied throughout the paper. Indeed,
\begin{align*}
    \FB_S(x^ay^b)&=|S|-z_S(a,b), \text{ and therefore}\\
    \FB_S(\mathcal C(S,\mathcal M))&=
    |S|-z_S(\mathcal M).
\end{align*}
Moreover, denoting by $w_H(\cdot)$ the Hamming weight and $d(\mathcal C(S,\mathcal M))$ the minimum distance of the code, we have
\begin{align*}
    w_H(\ev_S(f)) &= |S|-|Z_S(f)|\geq\FB_S(f) \text{ and }\\
    d(\mathcal C(S,\mathcal M)) &\geq \FB_S(\mathcal C(S,\mathcal M)).
\end{align*}
If $\mathcal M$ is decreasing, the minimum distance $d(\mathcal C(S,\mathcal M))$ is given by the footprint; see~\cite[Theorem 3.8]{LRV14}). Thus, 
\[
d(\mathcal C(S,\mathcal M))=
|S|-z_S(\mathcal M).
\]
Therefore, the set of minimum-weight codewords, denoted by $\MW(\mathcal C(S,\mathcal M))$, corresponds exactly to polynomials with the maximum possible number of zeros on the grid. In other words,
\[
\MW(\mathcal C(S,\mathcal M)) =
\{\ev_S(f):f\in Z_S(\mathcal M)\}.
\]

\subsection{One maximal monomial}
We first consider the situation where the support has a unique maximal monomial.

\begin{theorem}\label{thm.countonegen}
Take
$\mathcal M=\{x^ay^b:0\le a\le a_1,\ 0\le b\le b_2\}
\subseteq \Delta(I_{\F_q^2})$,
where $0\le a_1,b_2\le q-1$. Then
\[
\MW(\mathcal C(\F_q^2,\mathcal M))
=
\{\ev(f):f\in \Pi_{\F_q^2}(\mathcal M)\},
\]
and consequently
\[
|\MW(\mathcal C(\F_q^2,\mathcal M))|
=
(q-1)\binom{q}{a_1}\binom{q}{b_2}.
\]
\end{theorem}

\begin{proof}
It follows from Proposition~\ref{cor.onemon}.
\end{proof}

Therefore, when $|\max(\mathcal M)|=1$, every minimum-weight codeword arises from the expected family of coordinate factorizations.

\subsection{Two maximal monomials}
We now consider the first nontrivial case. Assume
\[
\max(\mathcal M)=
\{x^{a_1}y^{b_1},x^{a_2}y^{b_2}\}.
\]
Now, write
$\mathcal M=\mathcal M_1\cup\mathcal M_2$,
where
\[
\mathcal M_i
=\{x^ay^b: 0\le a\le a_i,\; 0\le b\le b_i\}.
\]
If $z_{\F_q^2}(a_1,b_1) \neq z_{\F_q^2}(a_2,b_2)$, then Theorem~\ref{thm.diffoot} shows that no additional minimum-weight codewords occur beyond those coming from the dominant rectangle. The interesting situation, which appears when $z_{\F_q^2}(a_1,b_1) = z_{\F_q^2}(a_2,b_2)$, is solved in Section~\ref{sec:reduction} by showing that the problem reduces to a decreasing set whose maximal monomials are $\{x^{d_x},y^{d_y}\}$.

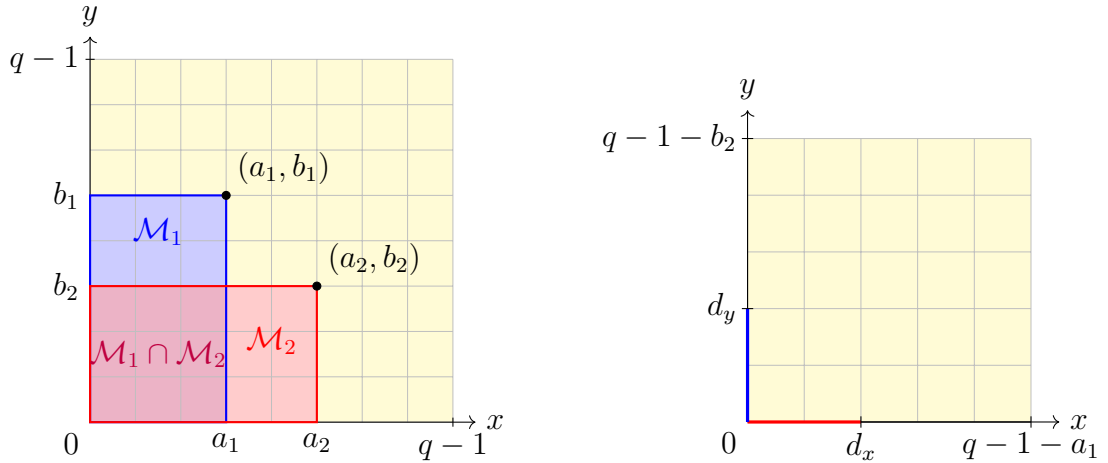
\begin{figure}[ht]
    \centering
    \begin{subfigure}[t]{0.48\textwidth}
        \centering
        \begin{tikzpicture}[scale=0.6]
            \fill[yellow!20] (0,0) rectangle (8,8);
            \fill[blue!20] (0,0) rectangle (3,5);
            \fill[red!20] (0,0) rectangle (5,3);
            \fill[purple!30] (0,0) rectangle (3,3);

            \draw[gray!50, very thin] (0,0) grid (8,8);

            \draw[->] (0,0) -- (8.5,0) node[right] {$x$};
            \draw[->] (0,0) -- (0,8.5) node[above] {$y$};

            \draw[blue, thick] (0,0) rectangle (3,5);
            \draw[red, thick]  (0,0) rectangle (5,3);

            \node[below] at (3,0) {$a_1$};
            \node[below] at (5,0) {$a_2$};
            \node[below] at (8,0) {$q-1$};

            \node[left] at (0,3) {$b_2$};
            \node[left] at (0,5) {$b_1$};
            \node[left] at (0,8) {$q-1$};

            \filldraw[black] (3,5) circle (2.5pt)
                node[above right] {$(a_1,b_1)$};

            \filldraw[black] (5,3) circle (2.5pt)
                node[above right] {$(a_2,b_2)$};

            \node[blue]   at (1.5,4.2) {$\mathcal M_1$};
            \node[red]    at (4.0,1.8) {$\mathcal M_2$};
            \node[purple] at (1.5,1.5) {$\mathcal M_1\cap\mathcal M_2$};

            \node[below left] at (0,0) {$0$};

            \draw (8,0.10) -- (8,-0.10);
            \draw (0.10,8) -- (-0.10,8);
        \end{tikzpicture}
        \caption{Original configuration of $\mathcal M_1$ and $\mathcal M_2$.}
        \label{fig:cambio-coord-izq}
    \end{subfigure}
    \hfill
    \begin{subfigure}[t]{0.48\textwidth}
        \centering
        \begin{tikzpicture}[scale=0.75]
            \def\dx{2}
            \def\dy{2}
            \def\nxm{5}
            \def\nym{5}

            \fill[yellow!20] (0,0) rectangle (\nxm,\nym);
            \draw[gray!50, very thin] (0,0) grid (\nxm,\nym);

            \draw[->] (0,0) -- (\nxm+0.5,0) node[right] {$x$};
            \draw[->] (0,0) -- (0,\nym+0.5) node[above] {$y$};

            \draw[red, very thick] (0,0) -- (\dx,0);
            \draw[blue, very thick] (0,0) -- (0,\dy);

            \draw[black] (\dx,0) -- (\nxm,0);
            \draw[black] (0,\dy) -- (0,\nym);

            \draw (\dx,0.10) -- (\dx,-0.10);
            \draw (\nxm,0.10) -- (\nxm,-0.10);
            \draw (0.10,\dy) -- (-0.10,\dy);
            \draw (0.10,\nym) -- (-0.10,\nym);

            \node[below left] at (0,0) {$0$};

            \node[below] at (\dx,0) {$d_x$};
            \node[below] at (\nxm,0) {$q-1-a_1$};

            \node[left] at (0,\dy) {$d_y$};
            \node[left] at (0,\nym) {$q-1-b_2$};

        \end{tikzpicture}
        \caption{Reduced support with $\max(\mathcal M')=\{x^{d_x},y^{d_y}\}$.}
        \label{fig:cambio-coord-der}
    \end{subfigure}

    \caption{
    Reduction from a decreasing set
    $\mathcal M=\mathcal M_1\cup\mathcal M_2$
    with two maximal monomials of equal footprint value
    to a decreasing set $\mathcal M'$
    satisfying
    $\max(\mathcal M')=\{x^{d_x},y^{d_y}\}$.
    }
    \label{fig:cambio-coordenadas}
\end{figure}

The next theorem translates this reduction into the language of minimum-weight codewords.

\begin{theorem}\label{thm.totalcount}
Assume $\max(\mathcal M)=\{x^{a_1}y^{b_1},x^{a_2}y^{b_2}\}$ with $a_1<a_2$, $b_1>b_2$, and $ z_{\F_q^2}(a_1,b_1)=z_{\F_q^2}(a_2,b_2)$.
Write $\mathcal M=\mathcal M_1\cup\mathcal M_2$, where $\mathcal M_i=\{x^ay^b: 0\le a\le a_i,\; 0\le b\le b_i\}$. Define $\mathcal M'=\{1,x,\ldots,x^{d_x}\}\cup\{1,y,\ldots,y^{d_y}\}$ with $d_x=a_2-a_1$ and $d_y=b_1-b_2$. Then,
\[
\begin{aligned}
\MW(\mathcal C(\F_q^2,\mathcal M))={}&
\MW(\mathcal C(\F_q^2,\mathcal M_1))\cup
\MW(\mathcal C(\F_q^2,\mathcal M_2))\\
&\cup\left\{\ev\!
\left(h\prod_{\alpha\notin X}(x-\alpha)\prod_{\beta\notin Y}(y-\beta)\right)
:(X,Y,h)\in\mathcal T(M')\right\}.
\end{aligned}
\]
Moreover,
$\bigl|\MW(\mathcal C(\F_q^2,\mathcal M))\bigr|
=\bigl|\MW(\mathcal C(\F_q^2,\mathcal M_1))\cup
\MW(\mathcal C(\F_q^2,\mathcal M_2))\bigr|+
|\mathcal T(M')|$.
\end{theorem}

\begin{proof}
We have
$\MW(\mathcal C(\F_q^2,\mathcal M))
=\{\ev(f):f\in Z_{\F_q^2}(\mathcal M)\}$.
Take
$f\in Z_{\F_q^2}(\mathcal M)$.
If $|in(f)\cap\max(\mathcal M)|=1$, then, by Proposition~\ref{prop.1init},
$f\in\Pi_{\F_q^2}(\mathcal M)$. In particular, we obtain
$\ev(f)\in\MW(\mathcal C(\F_q^2,\mathcal M_i))$
for some $i\in\{1,2\}$.

By Proposition~\ref{prop.factorh}, if
$|in(f)\cap\max(\mathcal M)|=2$, there exist
$X,Y\subseteq\F_q$, with $|X|=q-a_1$ and $|Y|=q-b_2$, and 
$h\in Z_{X\times Y}(\mathcal M')$
such that
\[
f=h \prod_{\alpha\notin X}(x-\alpha) \prod_{\beta\notin Y}(y-\beta).
\]
If $h\in\Pi_{X\times Y}(\mathcal M')$, then
$f\in\Pi_{\F_q^2}(\mathcal M)$. Therefore,
$\ev(f)\in\MW(\mathcal C(\F_q^2,\mathcal M_1))
\cup\MW(\mathcal C(\F_q^2,\mathcal M_2))$.
Otherwise, $(X,Y,h)\in\mathcal T(M')$.
Conversely, if $(X,Y,h)\in\mathcal T(M')$, then
\[
f=h\prod_{\alpha\notin X}(x-\alpha)\prod_{\beta\notin Y}(y-\beta)
\in Z_{\F_q^2}(\mathcal M),
\]
and consequently $\ev(f)\in\MW(\mathcal C(\F_q^2,\mathcal M))$.
This proves the result.

\end{proof}

Theorem~\ref{thm.totalcount} shows that every minimum-weight codeword arises either from one of the two extremal rectangles or from a nontrivial extremal polynomial on the reduced support. As an immediate consequence, we obtain a necessary condition for the existence of additional minimum-weight codewords.

\begin{corollary}
With the notation from Theorem~\ref{thm.totalcount}, if
\[
\MW(\mathcal C(\F_q^2,\mathcal M))\neq
\MW(\mathcal C(\F_q^2,\mathcal M_1))\cup
\MW(\mathcal C(\F_q^2,\mathcal M_2)),
\]
then
\[
\frac{q-a_1}{d_x}
=
\frac{q-b_2}{d_y}
\in\mathbb Z.
\]
\end{corollary}

\begin{proof}
It follows from Proposition~\ref{prop:noPolyexist}.
\end{proof}

\begin{remark}\rm
The converse is, however, not true, as Example~\ref{ex.11} shows. A complete characterization of when
\[
\MW(\mathcal C(\F_q^2,\mathcal M))=
\MW(\mathcal C(\F_q^2,\mathcal M_1))\cup
\MW(\mathcal C(\F_q^2,\mathcal M_2))
\]
is still open.
\end{remark}

Therefore, the classification of minimum-weight codewords is reduced to the determination of the set $\mathcal T(M')$ appearing in Theorem~\ref{thm.totalcount}. This is precisely the extremal-zero problem studied in Sections~\ref{sec:structure} and~\ref{sec:quadratic}.

\subsection{Quadratic cases}
The general reduction becomes completely explicit when
\[
d_x,d_y\le 2.
\]

Combining the results of Section~\ref{sec:quadratic} with Theorem~\ref{thm.totalcount}, we obtain exact formulas for the number of minimum-weight codewords. 

The following theorem addresses the first family.
\begin{theorem}
Let $\mathcal M=\mathcal M_1\cup\mathcal M_2$ be decreasing sets with
\[
\max(\mathcal M_1)=\{x^{a+2}y^b\},\qquad
\max(\mathcal M_2)=\{x^{a}y^{b+1}\},
\]
for some $a<q-2$ and $b<q-1$ satisfying
$t=q-b=\frac{q-a}{2}\in\mathbb Z$. Then,
\[
\left|\MW\bigl(\mathcal C(\F_q^2,\mathcal M)\bigr)\right|
=(q-1)
\begin{cases}
\displaystyle
\binom{q}{a+2}\binom{q}{b}+\binom{q}{a}\binom{q}{b+1}
+(q-1)^2q\binom{q/2}{t},\\
\hfill \text{if }\mathrm{char}(\F_q)=2,\\[8pt]
\displaystyle
\binom{q}{a+2}\binom{q}{b}+\binom{q}{a}\binom{q}{b+1}
+(q-1)q^2\binom{(q-1)/2}{t},\\
\hfill \text{if }\mathrm{char}(\F_q)\neq 2.
\end{cases}
\]
\end{theorem}

\begin{proof}
The result follows from Propositions~\ref{prop.char2},
\ref{prop.charneq2} and Theorems~\ref{thm.countonegen}
and~\ref{thm.totalcount}.
\end{proof}

The second family corresponds to the symmetric configuration.
\begin{theorem}
Let $\mathcal M=\mathcal M_1\cup\mathcal M_2$ be decreasing sets with
\[
\max(\mathcal M_1)=\{x^{a+2}y^a\},\qquad
\max(\mathcal M_2)=\{x^ay^{a+2}\},
\]
for some $a<q-2$ satisfying $t=\frac{q-a}{2}\in\mathbb Z$. Then,
\[
\left|\MW\bigl(\mathcal C(\F_q^2,\mathcal M)\bigr)\right|
=2(q-1)\binom{q}{a}\binom{q}{a+2}+\Gamma(q,t),
\]
where
\[
\Gamma(q,t)=
\begin{cases}
\displaystyle
\frac{(q-1)^3q}{2}\left[\binom{q/2}{t}+
2(q-2)\binom{q/4}{t}\right],
&q\equiv 0\pmod 2,\\[12pt]
\displaystyle
\frac{q^2(q-1)^2}{2}\binom{\frac{q-1}{2}}{t}+
\frac{q^2(q-1)^3}{4}
\left[\binom{\frac{q-5}{4}}{t}+
3\binom{\frac{q-1}{4}}{t}\right],
&q\equiv 1\pmod 4,\\[12pt]
\displaystyle
\frac{q^2(q-1)^2}{2}\binom{\frac{q-1}{2}}{t}+
\frac{q^2(q-1)^3}{4}\left[
3\binom{\frac{q-3}{4}}{t}+\binom{\frac{q+1}{4}}{t}
\right],
&q\equiv 3\pmod 4.
\end{cases}
\]
\end{theorem}
\begin{proof}
It follows from Theorems~\ref{thm.char2}, \ref{thm.charodd},
\ref{thm.countonegen}, and~\ref{thm.totalcount}.
\end{proof}

These formulas completely determine the minimum-weight codewords in the quadratic cases considered in this paper.

The results obtained in this section illustrate how the study of extremal polynomials on Cartesian grids naturally translates into coding-theoretic information. The reduction theory of Sections~\ref{sec:reduction} and~\ref{sec:structure} identifies configurations that may produce nontrivial minimum-weight codewords, while the explicit analysis of Section~\ref{sec:quadratic} yields exact counting formulas in the first nontrivial degrees.

More generally, Theorem~\ref{thm.totalcount} shows that determining minimum-weight codewords in decreasing codes with two maximal monomials is equivalent to understanding the corresponding extremal-zero problem. Thus, the structural results of this paper provide a framework for studying minimum-weight codewords beyond the classical Reed--Muller setting.

\section*{Declaration of AI}
ChatGPT was used for editorial purposes. After using this tool, the authors reviewed and edited the content as needed, and the authors remain responsible for all results.

\bibliography{ref}
\bibliographystyle{abbrv}

\end{document}